\documentclass[12pt]{amsart}
\usepackage[margin=1.2in]{geometry}
\usepackage{graphicx,latexsym}
\usepackage{comment}
\usepackage{tikz}
\usepackage{mathrsfs}
\usepackage{amsfonts, amssymb, amsmath, amsthm, amsxtra, amscd, hyperref}

\usepackage{mathtools}
\usepackage{todonotes}
\usepackage[foot]{amsaddr}
\usepackage{comment}

\usepackage{tikz}
\usetikzlibrary{matrix,arrows}
\allowdisplaybreaks

\newcommand{\N}{\mathbb{N}}

\newcommand{\R}{\mathbb{R}}
\newcommand{\C}{\mathbb{C}}

\newcommand{\dx}{{\rm d}x }

\newcommand{\supp}{\operatorname{supp}\,}
\newcommand{\singsupp}{\operatorname{sing\,supp}\,}

\newtheorem{theorem}{Theorem}[section]
\newtheorem{proposition}[theorem]{Proposition}
\newtheorem{lemma}[theorem]{Lemma}
\newtheorem{corollary}[theorem]{Corollary}

\theoremstyle{definition}
\newtheorem{definition}[theorem]{Definition}

\theoremstyle{remark}
\newtheorem{remark}[theorem]{Remark}

\numberwithin{equation}{section}

\makeatletter
\DeclareRobustCommand\widecheck[1]{{\mathpalette\@widecheck{#1}}}
\def\@widecheck#1#2{%
    \setbox\z@\hbox{\m@th$#1#2$}%
    \setbox\tw@\hbox{\m@th$#1%
       \widehat{%
          \vrule\@width\z@\@height\ht\z@
          \vrule\@height\z@\@width\wd\z@}$}%
    \dp\tw@-\ht\z@
    \@tempdima\ht\z@ \advance\@tempdima2\ht\tw@ \divide\@tempdima\thr@@
    \setbox\tw@\hbox{%
       \raise\@tempdima\hbox{\scalebox{1}[-1]{\lower\@tempdima\box
\tw@}}}%
    {\ooalign{\box\tw@ \cr \box\z@}}}
\makeatother

\usetikzlibrary{arrows,chains,matrix,positioning,scopes}
\makeatletter
\tikzset{join/.code=\tikzset{after node path={%
\ifx\tikzchainprevious\pgfutil@empty\else(\tikzchainprevious)%
edge[every join]#1(\tikzchaincurrent)\fi}}}
\makeatother
\tikzset{>=stealth',every on chain/.append style={join},
         every join/.style={->}}
\tikzstyle{labeled}=[execute at begin node=$\scriptstyle,
   execute at end node=$]

\begin{document}
	
\title[Quantitative Runge type approximation theorems]{Quantitative Runge type approximation theorems for zero solutions of certain partial differential operators}
	
\author[A. Debrouwere]{A. Debrouwere$^1$}
\address{$^1$Department of Mathematics and Data Science, Vrije Universiteit Brussel, Pleinlaan 2, 1050 Brussels, Belgium}
\email{andreas.debrouwere@vub.be}

\author[T.\ Kalmes]{T.\ Kalmes$^2$}
\address{$^2$Faculty of Mathematics, Chemnitz University of Technology, 09107 Chemnitz, Germany}
\email{thomas.kalmes@math.tu-chemnitz.de}

\begin{abstract}
	We prove  quantitative Runge type approximation results for spaces of smooth zero solutions of several classes of linear partial differential operators with constant coefficients. Among others, we establish such  results for arbitrary operators on convex sets, elliptic operators, parabolic operators, and the wave operator in one spatial variable. Our methods are inspired by the study of linear topological invariants for kernels of partial differential operators.  As a part of our work, we also show a  qualitative Runge type approximation theorem for subspace elliptic operators, which seems to be new and of independent interest. \\
	
	\noindent Keywords: Quantitative Runge type approximation theorems; Quantitative Lax-Malgrange theorem; Partial differential operators; Linear topological invariants for kernels of differential operators\\
	
	\noindent MSC 2020: 35A35, 35E20, 46A63
	
\end{abstract}

\maketitle

\section{Introduction and statement of the main results}

A well-known consequence of Runge's classical theorem on  approximation by rational functions is that for open subsets $Y\subseteq X$ of the complex plane $\mathbb{C}$, every holomorphic function on $Y$ can be approximated uniformly on compact subsets by holomorphic functions on $X$ if and only if  $X$ does not contain a compact connected component of $\mathbb{C}\backslash Y$. This result has been generalized independently by Lax \cite{Lax} and Malgrange \cite{Malgrange} to kernels of elliptic constant coefficient differential operators. Browder \cite{Browder1962b, Browder1962}  thoroughly investigated Runge type approximation theorems for elliptic differential operators with variable coefficients in various function and distribution spaces. 

For non-elliptic differential operators much less is known.  Given any non-zero constant coefficient partial differential operator $P(D)$, Malgrange \cite[Chapitre 1.2, Th\'eor\`eme 2]{Malgrange} (see also  \cite[Theorem 10.5.1]{HoermanderPDO2}) showed that the linear span of the exponential solutions of $P(D)$ is dense in the space of smooth zero solutions of $P(D)$ on a  convex set. H\"ormander \cite[Theorem 10.5.2]{HoermanderPDO2} and Tr\`eves \cite[Theorem 26.1]{Treves1967-2} proved implicit approximation results for general constant coefficient partial differential operators that are applicable to arbitrary open sets $Y, X$ and to $P$-convex open sets $Y,X$, respectively. However, geometrical conditions on  $Y,X$ ($Y$ non-convex) ensuring the validity of a Runge type approximation theorem are only known for special types of operators. In \cite{Diaz1980,GT,Jones1975} approximation theorems for the heat operator have been shown. The second author \cite{Kalmes21} recently extended these results to constant coefficient differential operators  with a single characteristic direction, including the Schr\"odinger operator and parabolic operators like the heat operator. Runge type approximation theorems  for variable coefficient  parabolic differential operators of second order have been lately obtained in \cite{EnGaPe19}. 

In the past few years, there has been a considerable interest in \emph{quantitative} approximation results. Roughly speaking, the term quantitative  here means that given an approximation error  $\varepsilon>0$  and a solution $f$ to a differential equation on a certain set, the size of the approximant $h_\varepsilon$, which is a solution to the same differential equation on a larger set, can be estimated in terms of $\varepsilon$ and the size of $f$ (where the size of a function is measured in terms of a scale of seminorms). The goal is thus to determine the cost of approximation in the context of Runge type theorems.  In their seminal work \cite{RuSa20}, R\"uland and Salo proved quantitative  Runge type approximation theorems for  fractional Schr\"odinger operators. Shortly after this work, quantitative  approximation results were shown for various other types of operators, see \cite{RuSa19} for  elliptic variable coefficient operators of second order, \cite{RuSa20-1} for the fractional heat and wave operator, and \cite{EnPe21} for the Schr\"odinger operator. We remark that quantitative approximation results may be used to obtain stability results for various ill-posed inverse problems \cite{Ruland, RuSa19, RuSa20}. 

Inspired by the above results, we establish in the present paper quantitative Runge type approximation results for  spaces of zero solutions of several classes of constant coefficient partial differential operators. Most notably, we consider arbitrary operators on convex sets, elliptic operators, parabolic operators, and the wave operator in one spatial variable. 

Let $P\in\C[\xi_1,\ldots,\xi_d]$ be a polynomial and consider the corresponding constant coefficient differential operator $P(D)=P(-i\frac{\partial}{\partial x_1},\ldots,-i\frac{\partial}{\partial x_d})$. Let $X \subseteq \R^d$ be open. We write $\mathcal{E}(X)$ for the space of smooth functions on $X$ endowed with its natural Fr\'echet space topology, i.e., the one induced by the family of seminorms
\begin{equation}
\label{norms-E}
\|f\|_{K,r}=\sup_{x\in K,|\alpha|\leq r}|\partial^\alpha f(x)|, \qquad K \subseteq X \mbox{ compact}, r \in \N_0.
\end{equation}
We set 
$$\mathscr{E}_P(X) = \{f\in\mathscr{E}(X)\,;\, P(D)f=0\}$$ 
and endow it with the subspace topology from $\mathscr{E}(X)$. A pair of open sets $Y\subseteq X \subseteq \R^d$ satisfies the (qualitative) Runge  approximation property with respect to the operator $P(D)$  if the restriction map
$$r^P_{\mathscr{E}}:\mathscr{E}_P(X)\rightarrow\mathscr{E}_P(Y), f\mapsto f_{|Y}$$
has dense range. Our goal is to give quantified versions of this property in terms of the seminorms defined in \eqref{norms-E}. If $P(D)$ is hypoelliptic, the topology of the space $\mathscr{E}_P(X)$ is generated by the sup-seminorms $ \|\, \cdot \, \|_{K} =  \|\, \cdot \, \|_{K,0}$, $K \subseteq X$ compact. Hence, in this case, it seems more  natural to look for quantitative approximation results with respect to the seminorms $ \|\, \cdot \, \|_{K}$. 

We now state a sample of our main results.  We start with a quantitative approximation result for arbitrary operators on convex sets.

\begin{theorem}\label{theo: quantitative convex}
Let $P\in \C[\xi_1,\ldots,\xi_d] \backslash \{0\}$. Let $K \subseteq \R^d$ be compact and convex. Let $Y\subseteq\R^d$ be open with $K \subseteq Y$. Then, for all $L \subseteq Y$ compact with $K \subseteq \operatorname{int} L$,  $M \subseteq \R^d$ compact, and $r_1,r_2 \in \N_0$ there are $s, C>0$ such that 
		\begin{gather*} 
		\forall f\in\mathscr{E}_P(Y), \varepsilon\in (0,1) \,\exists h_\varepsilon \in\mathscr{E}_P(\R^d) \, :\\ \nonumber
		\|f-h_\varepsilon\|_{K,r_1}\leq \varepsilon\|f\|_{L,r_1+1} \quad \mbox{ and } \quad \|h_\varepsilon\|_{M,r_2}\leq\frac{C}{\varepsilon^s}\|f\|_{L}.
	\end{gather*}
If $P(D)$ is hypoelliptic, we have that for all $L \subseteq Y$ compact with $K \subseteq \operatorname{int} L$ and $M \subseteq \R^d$ compact there are $s, C>0$ such that 
\begin{gather*} 
	\forall f\in\mathscr{E}_P(Y), \varepsilon\in (0,1) \,\exists h_\varepsilon \in\mathscr{E}_P(\R^d) \, :\\ \nonumber
	\|f-h_\varepsilon\|_{K}\leq \varepsilon\|f\|_{L} \quad \mbox{ and } \quad \|h_\varepsilon\|_{M}\leq\frac{C}{\varepsilon^s}\|f\|_{L}.
\end{gather*}
\end{theorem}

Next, we consider elliptic operators. An open  set $U \subseteq \R^d$ is said to have  \emph{continuous boundary} if for each $\xi \in \partial U$ there exists a homeomorphism $h$ from the closed unit ball $\overline{B}(0,1)$ in $\R^d$ onto a compact neighborhood $K$ of $\xi$ such that $h(0) = \xi$  
 and  $\overline{U} \cap K= h\left(\{x\in \overline{B}(0,1)\, ; \,x_d\leq 0\}\right)$. Obviously, every Lipschitz domain has continuous boundary. The next result may be considered as a quantified version of the Lax-Malgrange theorem.
 
 \begin{theorem}\label{theo: quantitative Lax-Malgrange}
	Let $P\in \C[\xi_1,\ldots,\xi_d]$ be elliptic and let $X\subseteq\R^d$ be open. Let $K$ be a compact subset of $X$ that is the closure of an open set with continuous boundary. Suppose that $X$ does not contain a bounded connected component of $\R^d\backslash K$. Let $Y \subseteq \R^d$ be open with $K\subseteq Y\subseteq X$.
	Then, for all $L \subseteq Y$ compact with $K \subseteq \operatorname{int} L$ and $M \subseteq X$ compact there are $s, C>0$ such that 
	\begin{gather*} 
	\forall f\in\mathscr{E}_P(Y), \varepsilon\in (0,1) \,\exists h_\varepsilon \in\mathscr{E}_P(X) \, :\\ \nonumber
	\|f-h_\varepsilon\|_{K}\leq \varepsilon\|f\|_{L} \quad \mbox{ and } \quad \|h_\varepsilon\|_{M}\leq\frac{C}{\varepsilon^s}\|f\|_{L}.
\end{gather*}
\end{theorem}
\noindent Theorem  \ref{theo: quantitative Lax-Malgrange} for the Cauchy-Riemann operator is due to Petzsche \cite[Theorem 4.2 (a)]{Petzsche1980} (he did not need to assume that $U$ has continuous boundary). To the best of our knowledge, Petzsche's result seems to be the first quantitative Runge  type approximation result. Furthermore, Theorem  \ref{theo: quantitative Lax-Malgrange} complements \cite[Theorem 3]{RuSa19}, where a similar result is proven  for elliptic variable coefficient  operators of second order with respect to  Sobolev type norms.

In \cite[Corollary 4]{Kalmes21} an approximation result for certain $k$-parabolic  operators  on tubular domains is shown. We now give a quantified version of it.  
We denote the elements of $ \R^d = \R^{n+1}$ by $(t,x)$ with $t\in\R$, $x\in\R^n$.

\begin{theorem}\label{theo: quantitative parabolic}
	Let $Q\in\C[\xi_1,\ldots,\xi_n]$ be an elliptic polynomial of degree $m$ with principal part $Q_m$, let $r\in\N$ with $r<m$, and let $\alpha\in\C\backslash\{0\}$.
	Suppose that $\{ Q_m(x) \, ; \, x \in \R^n\backslash\{0\} \} \cap \{\alpha t^r\, ; \,t\in\R\} = \emptyset$. Consider the (hypoelliptic) differential operator
	$$P(D) = \alpha D^r_t -Q(D_x).$$
	Let $G\subseteq\R^n$ be open. Let $K$ be a compact subset of $G$ that is the closure of an open set with continuous boundary.
	Suppose  that $G$ does not contain a bounded connected component of $\R^n\backslash K$, and let $[t_1,t_2] \subseteq \R$ be an interval. Let $ H \subseteq \R^n$ be open with $K \subseteq H \subseteq G$ and let $J \subseteq \R$ be an open interval with $[t_1,t_2] \subseteq J$. Then, for all  $L \subseteq H$ compact with $K \subseteq \operatorname{int} L$, $\delta > 0$ with $[t_1-\delta,t_2+\delta]\subseteq J$,   and $M \subseteq \R \times G$ compact there are $s, C>0$ such that
	\begin{gather*} 
		\forall f\in\mathscr{E}_P(J \times H), \varepsilon\in (0,1) \,\exists h_\varepsilon \in\mathscr{E}_P(\R \times G) \, :\\ 
		\|f-h_\varepsilon\|_{[t_1,t_2] \times K}\leq \varepsilon\|f\|_{[t_1-\delta,t_2+\delta]\times L} \quad \mbox{ and } \quad \|h_\varepsilon\|_{M}\leq\frac{C}{\varepsilon^s}\|f\|_{[t_1-\delta,t_2+\delta]\times L}.
	\end{gather*}
\end{theorem}

\noindent Theorem \ref{theo: quantitative parabolic} is particularly applicable to parabolic operators like the heat operator.

We also consider the wave operator in one spatial variable. 

\begin{theorem}\label{theo: quantitative Runge for wave operator}
	Let $P(D) = \frac{\partial^2}{\partial_t^2}- \frac{\partial^2}{\partial_x^2}$ and let  $X\subseteq\R^2$ be open such that $X$ is $P$-convex for supports. Let $K \subseteq X$ be compact and let $Y \subseteq \R^2$ be open with $K \subseteq Y \subseteq X$ such that there is no  $x\in\R^2$ such that $X$ contains a compact connected component of one of the sets $(\R^2\backslash Y)\cap\{(x_1+s,x_2+s)\,;\,s\in\R\}$ or $(\R^2\backslash Y)\cap\{(x_1-s,x_2+s)\, ; \,s\in\R\}$. Then, for all $L \subseteq Y$ compact with $L=\overline{\mbox{int}(L)}$ and $K \subseteq \operatorname{int} L$  such that there is no $x\in\R^2$ such that $Y$ contains a bounded connected component of one of the sets $(\R^2\backslash L)\cap \{(x_1+s,x_2+s)\, ; \,s\in\R\}$, or $(\R^2\backslash L)\cap\{(x_1-s,x_2+s)\, ; \,s\in\R\}$, $M \subseteq X$ compact and $r_1,r_2 \in \N_0$ there are $s, C>0$ such that 
	\begin{gather*}
		\forall f\in\mathscr{E}_P(Y), \varepsilon\in (0,1) \,\exists h_\varepsilon \in\mathscr{E}_P(X) \, :\\
		\|f-h_\varepsilon\|_{K,r_1}\leq \varepsilon\|f\|_{L,r_1+1} \quad \mbox{ and } \quad \|h_\varepsilon\|_{M,r_2}\leq\frac{C}{\varepsilon^s}\|f\|_{L}.
	\end{gather*}
\end{theorem}

\noindent Theorems \ref{theo: quantitative convex}-\ref{theo: quantitative parabolic} are more natural than Theorem \ref{theo: quantitative Runge for wave operator} in the sense that in the former results we impose
geometric conditions on the compact set on which we approximate the given function $f$,  while in the latter we need to impose such conditions on its supersets $Y$ and $L$. We believe that an analogue of Theorem \ref{theo: quantitative parabolic} for the wave operator in one spatial variable is true but are unable to show this.

To the best of our knowledge, approximation results for the wave operator have not yet been studied. The proof of Theorem  \ref{theo: quantitative Runge for wave operator} depends heavily on the fact that the wave operator in one spatial variable can be factored into two first order operators. Therefore, our proof cannot be extended to the multidimensional setting.  It would be very interesting to find Runge type approximation results for the wave operator in several spatial variables.

As already stated above,  Malgrange \cite[Chapitre 1.2, Th\'eor\`eme 2]{Malgrange} (see also \cite[Theorem 10.5.1]{HoermanderPDO2}) showed that the linear span of the  exponential solutions  of $P(D)$ is always dense in $\mathscr{E}_P(\R^d)$. Moreover, by \cite[Chapitre 1.2, Th\'eor\`eme $2^\prime$]{Malgrange}, polynomial solutions of $P(D)$ are dense in $\mathscr{E}_P(\R^d)$ if and only if every non-constant factor of $P$ vanishes at the origin. Hence, in any of the above global approximation results, i.e., when $X=\R^d$, the approximating function $h_\varepsilon$ can always be chosen to be a linear combination of exponential solutions of $P(D)$. If every non-constant factor of $P$ vanishes at the origin, $h_\varepsilon$ can even be chosen to be a polynomial solution of $P(D)$.

Finally, we  would like to point out that the way we deduce quantitative approximation results from qualitative ones is different from the one used by  R\"uland and Salo in \cite{RuSa19, RuSa20}. Their method relies on a quantitative unique continuation result and techniques from spectral theory for Hilbert spaces (singular value decompositions) combined with a duality argument. Instead, we use here  an  \emph{internal} quantitative approximation result for the space $\mathscr{E}_P(X)$ (Proposition \ref{cor: explicit Omega}) to achieve this goal. The latter is shown in two steps: First, we prove an analogous result for the space of distributional zero solutions (Proposition \ref{lem: explicit POmega}) by combining  a result from \cite[Section 3.4.5]{Wengenroth} with abstract functional analytic arguments (Grothendieck's factorization theorem and the Schwartz kernel theorem). Hereafter, we deduce the desired result for the space of smooth zero solutions from it by a regularization procedure developed in our recent article \cite{DeKa22}.  This method is heavily inspired by  the study of linear topological invariants for kernels of partial differential operators \cite{B-D2006,DeKa22, Kalmes19,Vogt1983}, we explain this relation at the end of Section \ref{sec: prelim}.

The article is organized as follows. In Section \ref{sec: prelim} we discuss some preliminary notions and results.  We also give an outline of our strategy to deduce quantitative approximation results from qualitative ones. Section \ref{sec: qualitative Runge} is devoted to qualitative Runge type approximation theorems. We show an approximation result for subspace elliptic differential operators and apply it to obtain such a result for differential operators that factor into first order operators, including the wave operator in one spatial variable. The main technical core of our work is given in Section \ref{sec: technical}, where we prove a crucial auxiliary quantitative approximation result for the space $\mathscr{E}_P(X)$.  Finally, in Section \ref{sec: quantitative Runge}, we  establish quantitative Runge type approximation theorems for various types of operators. In particular, we show Theorems \ref{theo: quantitative convex}-\ref{theo: quantitative Runge for wave operator}.

\section{Preliminaries and outline of the proof strategy}\label{sec: prelim}

In this section we fix the notation and  recall various notions and results which will be used throughout. We also  sketch  our strategy to obtain quantitative Runge type approximation results. We use standard terminology from functional analysis \cite{Jarchow1981, M-V}, distribution theory \cite{HoermanderPDO1}, and linear partial differential operators \cite{HoermanderPDO2}.  

Let $X \subseteq \R^d$ be open.  We denote by $\mathscr{E}(X)$ the Fr\'echet space of smooth  functions on $X$, its topology is generated by the seminorms \eqref{norms-E}. Given $K \subseteq \R^d$ compact, we write $\mathscr{D}(K)$ for the space of smooth functions on $\R^d$ with support in $K$ and endow it with the subspace topology from $\mathscr{E}(\R^d)$. The space of test functions on $X$ is given by
$$
\mathscr{D}(X) = \varinjlim_{\tiny{K \subseteq X \mbox{ compact}}} \mathscr{D}(K).
$$
The space of distributions $\mathscr{D}'(X)$ is defined as the dual of $\mathscr{D}(X)$, we  endow it with its strong dual topology.

Let $P \in \C[\xi_1, \ldots, \xi_d]$ and consider the corresponding  differential operator $P(D)=P(-i\frac{\partial}{\partial x_1},\ldots,-i\frac{\partial}{\partial x_d})$. We set $\check{P}(\xi)=P(-\xi)$.  Let $P(\xi)=\sum_{|\alpha| \leq m}a_\alpha\xi^\alpha$, with $m$ the degree of $P$. We denote by $P_m$ the principal part of $P$, i.e., $P_m(\xi)=\sum_{|\alpha|=m}a_\alpha\xi^\alpha$. 
A hyperplane $\{x\in\R^d \,;\,\langle N,x\rangle=c\}$, $N \in\R^d\backslash\{0\}$, $c\in\R$, is called  a \emph{characteristic hyperplane} for $P(D)$ if  $P_m(N)=0$. The differential operator $P(D)$ defines a continuous linear self map on both $\mathscr{E}(X)$  and $\mathscr{D}'(X)$. We set 
$$\mathscr{E}_P(X)=\{f\in\mathscr{E}(X) \,;\, P(D)f=0\} \quad \mbox{and} \quad \mathscr{D}'_P(X)=\{f\in\mathscr{D}'(X) \,;\, P(D)f=0\},$$
and endow these spaces with the subspace topology from $\mathscr{E}(X)$ and $\mathscr{D}'(X)$, respectively. 

Malgrange \cite{Malgrange} (see also \cite[Section 10.6]{HoermanderPDO2}) showed  that $P(D): \mathscr{E}(X) \to\mathscr{E}(X)$ is surjective  if and only if $X$ is \emph{$P$-convex for supports}, i.e., if for every $K \subseteq X$ compact there is $L \subseteq X$ compact  such that for all $f\in\mathscr{E}'(X)$
$$\supp\check{P}(D)f \subseteq K\Longrightarrow\supp f \subseteq L.$$ 
H\"ormander \cite[Section 10.7]{HoermanderPDO2} showed that $P(D): \mathscr{D}'(X) \to \mathscr{D}'(X)$ is surjective if and only if $X$ is both $P$-convex for supports as well as \emph{$P$-convex for singular supports}, the latter meaning that for every $K \subseteq X$ compact there is $L \subseteq X$ compact  such that for all $f\in\mathscr{E}'(X)$
$$\singsupp\check{P}(D)f \subseteq K\Longrightarrow\singsupp f \subseteq L.$$ 
By the theorems of support and singular support \cite[Theorem 7.3.2 and Theorem 7.3.9]{HoermanderPDO1}, convex open sets $X$ are  $P$-convex for supports and singular supports for any $P\neq 0$. If $P$ is elliptic,  every $X \subseteq \R^d$ open is $P$-convex for supports and singular supports \cite[Corollary 10.8.2 and Theorem 11.1.1]{HoermanderPDO2}.

We define $P^+(\xi_1,\ldots,\xi_{d+1})=P(\xi_1,\ldots,\xi_d)$ and call the corresponding differential operator $P^+(D)$  the \emph{augmented operator} of $P(D)$. As explained at the end of this section, the surjectivity of $P^+(D): \mathscr{D}'(X\times\R) \to \mathscr{D}'(X\times\R)$ will be one of our basic assumptions for establishing  quantitative Runge type approximation results. Note that $P(D): \mathscr{D}'(X) \to \mathscr{D}'(X)$ is always surjective if $P^+(D): \mathscr{D}'(X\times\R) \to \mathscr{D}'(X\times\R)$ is so. The question whether  $P^+(D): \mathscr{D}'(X\times\R) \to \mathscr{D}'(X\times\R)$ is surjective if $P(D): \mathscr{D}'(X) \to \mathscr{D}'(X)$ is so, was posed by Bonet and Doma\'nski \cite[Proposition 8.1]{B-D2006} in connection with the problem of parameter dependence of surjective differential operators. In \cite{Kalmes12-3}  the second author showed that for $d=2$ the augmented operator $P^+(D): \mathscr{D}'(X\times\R) \to \mathscr{D}'(X\times\R)$   of a surjective  differential operator $P(D): \mathscr{D}'(X) \to \mathscr{D}'(X)$   is always surjective, while, for $d \geq 3$,  he constructed in \cite{Kalmes12-2} a (hypoelliptic) differential operator and an open set $X \subseteq \R^d$ such that $P(D): \mathscr{D}'(X) \to \mathscr{D}'(X)$  is surjective but $P^+(D): \mathscr{D}'(X\times\R) \to \mathscr{D}'(X\times\R)$  is not. 
 
Next, we give some positive results concerning the surjectivity of the augmented operator for certain classes of differential operators \cite{Kalmes12-3, Kalmes19}. To this end,  recall that a function $f:X\rightarrow[0,\infty]$, $X\subseteq\R^d$ open, is said to satisfy the \emph{minimum principle} in a closed set $F\subseteq\R^d$ if for every compact subset $K$ of $X\cap F$ it holds that
$$\inf_{x\in K}f(x)=\inf_{x\in\partial_F K}f(x),$$
where $\partial_F K$ denotes the boundary of $K$ in $F$.  The  \emph{boundary distance} of $X$ is given by the  function 
$$
d_X:X\rightarrow [0,\infty], \quad d_X(x)=\inf\{|x-y| \, ; \, y\in\R^d\backslash X\}.
$$ 
A polynomial $P\in\C[\xi_1,\ldots,\xi_d]$ is said to \emph{act along a  subspace} $W\subseteq\R^d$ if $P(\xi)=P(\pi_W\xi)$ for all $\xi\in\R^d$, where $\pi_W$ denotes the orthogonal projection  onto $W$. A polynomial $P$ which acts along a subspace $W$ is said to be \emph{elliptic on} $W$ if its principal part $P_m$ satisfies $P_m(\xi)\neq 0$ for every $\xi\in W\backslash\{0\}$. In such a case, $P(D)$ is called \emph{subspace elliptic}. For example, the augmented operator of an elliptic operator is subspace elliptic.  We also mention that for a semi-elliptic polynomial $P$ with principal part $P_m$ the set $\{P_m=0\}=\{\xi\in\R^d \,; \, P_m(\xi)=0\}$ is a subspace of $\R^d$ (see e.g.\ \cite[Proposition 2(iii)]{FrKa}).
 
 \begin{theorem}\label{theo: cases of Omega}
 	Let $P\in\C[\xi_1,\ldots,\xi_d]$ and let $X\subseteq\R^d$ be open.
 	\begin{itemize}
 		\item[(a)] \cite[Theorem 9 and Theorem 16]{Kalmes19} Suppose that $P$ acts along the subspace $W$ of $\R^d$ and is elliptic there. Then, the following statements are equivalent:
 		\begin{itemize}
 			\item[(i)] $P(D):\mathscr{E}(X)\rightarrow\mathscr{E}(X)$ is surjective.
 			\item[(ii)] $P(D):\mathscr{D}'(X)\rightarrow\mathscr{D}'(X)$ is surjective.
 			\item[(iii)] $P^+(D):\mathscr{D}'(X\times\R)\rightarrow\mathscr{D}'(X\times\R)$ is surjective.
 			\item[(iv)] The boundary distance $d_X$ satisfies the minimum principle in the affine subspace $x+W$ for every $x\in\R^d$.
 		\end{itemize}
 		\item[(b)] \cite[Corollary 5 and Theorem 16]{Kalmes19} Suppose that $P$ is semi-elliptic with principal part $P_m$ such that $\{P_m=0\}$ is one-dimensional. Then, the following statements are equivalent:
 		\begin{itemize}
 			\item[(i)] $P(D):\mathscr{E}(X)\rightarrow\mathscr{E}(X)$ is surjective.
 			\item[(ii)] $P(D):\mathscr{D}'(X)\rightarrow\mathscr{D}'(X)$ is surjective.
 			\item[(iii)] $P^+(D):\mathscr{D}'(X\times\R)\rightarrow\mathscr{D}'(X\times\R)$ is surjective.
 			\item[(iv)] The boundary distance $d_X$ satisfies the minimum principle in every characteristic hyperplane for $P(D)$.
 		\end{itemize}
 		\item[(c)] \cite[Theorem 21]{Kalmes12-3} Assume that $d=2$. Then, the following statements are equivalent:
		\begin{itemize}
 			\item[(i)] $P(D):\mathscr{E}(X)\rightarrow\mathscr{E}(X)$ is surjective.
 			\item[(ii)] $P(D):\mathscr{D}'(X)\rightarrow\mathscr{D}'(X)$ is surjective.
 			\item[(iii)] $P^+(D):\mathscr{D}'(X\times\R)\rightarrow\mathscr{D}'(X\times\R)$ is surjective.
 			\item[(iv)] The boundary distance $d_X$ satisfies the minimum principle in every characteristic line for $P(D)$.
 		\end{itemize}
 	\end{itemize}
 \end{theorem}
 
\begin{remark} \label{rem:augmented-elliptic}
	Theorem \ref{theo: cases of Omega}(a) with $W= \R^d$ impies that $P^+(D):\mathscr{D}'(X\times\R)\rightarrow\mathscr{D}'(X\times\R)$ is always surjective if $P(D)$ is elliptic. This result is essentially due to Vogt \cite{Vogt1983}. 
\end{remark}

We now outline our strategy to obtain quantitative Runge type approximation results. In particular, we explain how our method is inspired by and connected to the linear topological invariants $(\Omega)$ for $\mathscr{E}_P(X)$ and  $(P\Omega)$ for $\mathscr{D}'_P(X)$. We refer to \cite{B-D2006,DeKa22,M-V, Kalmes19,Vogt1983} for more information on these conditions. The space $\mathscr{E}_P(X)$ satisfies $(\Omega)$ if for all  $K\subseteq X$ compact and  $k\in\N_0$ there are $L = L(K,k)\subseteq X$ compact and  $l=l(K,k)\in\N_0$ such that for all $M\subseteq X$ compact and  $m\in\N_0$ there are $C,s >0$ such that
\begin{gather}
	\label{def-Omega}
	\forall g\in\mathscr{E}_P(X), \varepsilon\in (0,1) \, \exists h_\varepsilon\in\mathscr{E}_P(X):\\ \nonumber
	\|g-h_\varepsilon\|_{K,k}\leq \varepsilon\|g\|_{L,l} \quad \mbox{ and } \quad \|h_\varepsilon\|_{M,m}\leq\frac{C}{\varepsilon^s}\|g\|_{L,l}.
\end{gather}
Condition $(\Omega)$ may be interpreted as an internal quantitative approximation result for the space $\mathscr{E}_P(X)$. Roughly speaking, the condition $(P\Omega)$ for $\mathscr{D}'_P(X)$ is an analogue of the condition $(\Omega)$ for $\mathscr{D}'_P(X)$ adapted to the linear topological structure of this space\footnote{$\mathscr{D}'_P(X)$ is a so-called  $(PLS)$-space \cite{B-D2006} and not metrizable.}, see \cite[Section 5]{B-D2006}  and \cite{DeKa22} for the precise definition.

 Now let $Y \subseteq X$ be open and suppose that the restriction map 
\begin{equation}
	\label{restriction}
 r_{\mathscr{E}}^P:\mathscr{E}_P(X)\rightarrow\mathscr{E}_P(Y), \, f \mapsto f_{\mid Y}
\end{equation}
 has dense range. Furthermore, assume that $\mathscr{E}_P(X)$ satisfies $(\Omega)$. Let $K\subseteq Y$ be compact and  $k\in\N_0$ be arbitrary. Assume that one can choose the compact $L(K,k)$ in the condition $(\Omega)$ inside  $Y$.  Let also $l(K,k)$ be as above. Then, a simple argument (cf. the proof of Proposition \ref{prop: general quantitative} below) shows that  for all $M\subseteq X$ compact and  $m\in\N_0$ there are $C,s >0$ such that
 \begin{gather*}
 	\forall f\in\mathscr{E}_P(Y), \varepsilon\in (0,1) \, \exists h_\varepsilon\in\mathscr{E}_P(X):\\
 	\|f-h_\varepsilon\|_{K,k}\leq \varepsilon\|f\|_{L,l} \quad \mbox{ and } \quad \|h_\varepsilon\|_{M,m}\leq\frac{C}{\varepsilon^s}\|f\|_{L,l},
 \end{gather*}
 which is a quantitative Runge type approximation result. Of course, this result is only interesting if we have information on how we can choose $L = L(K,k)$ and $l =l(K,k)$ in terms of the given $K$ and $k$ (we want $L$ and $l$ to be as close as possible to $K$ and $k$, respectively).  However, there does not seem to exist a method to show  $(\Omega)$ for $\mathscr{E}_P(X)$ that allows one to keep track of  $L(K,k)$ and $l(K,k)$ in terms of $K$ and $k$, even for specific type of operators (e.g.\ elliptic operators) or sets (e.g. convex sets).  The development of such a method 
 is one of the main novelties in our work.  We now give the basic idea of our approach to study this problem, proofs are given  in Section \ref{sec: technical}.   A  seminal result of Bonet and Doman\'ski \cite[Proposition 8.1]{B-D2006} asserts that $\mathscr{D}'_P(X)$ satisfes $(P\Omega)$ if and only if the augmented operator $P^+(D): \mathscr{D}'(X\times\R) \to \mathscr{D}'(X\times\R)$ is surjective. In our recent work \cite{DeKa22}, we showed that, if $P(D): \mathscr{D}'(X) \to \mathscr{D}'(X)$ is surjective, then  $\mathscr{D}'_P(X)$ satisfes $(P\Omega)$ if and only if  $\mathscr{E}_P(X)$ satisfes $(\Omega)$. Inspired by these results, we will tackle the above problem in two steps. First, we establish a quantitative approximation result (with control on the compacts occurring in it) for distributional zero solutions of $P(D)$ under the assumption that  $P^+(D): \mathscr{D}'(X\times\R) \to \mathscr{D}'(X\times\R)$ is surjective (Proposition \ref{lem: explicit POmega}). Next, we  use the same method  as in \cite{DeKa22}  to deduce from this result information on how one can choose $L(K,k)$ and $l(K,k)$ occurring in $(\Omega)$ in  terms of $K$ and $k$ (Propositon \ref{cor: explicit Omega}).

Summarizing, we will be able to deduce a quantitative approximation result from a qualitative  one (meaning that the restriction map \ref{restriction} has dense range) under the assumption  that the augmented operator $P^+(D): \mathscr{D}'(X\times\R) \to \mathscr{D}'(X\times\R)$ is surjective. 

\section{Qualitative Runge type approximation results}\label{sec: qualitative Runge}

The purpose of this section is to show a qualitative Runge type approximation theorem  for subspace elliptic differential operators.  As a consequence, we will give an approximation result for differential operators that factor into first order operators, including the  wave operator in one spatial variable. Furthermore, we recall a  Runge type approximation theorem for semi-elliptic operators with a single characteristic direction from \cite{Kalmes21}.

We will use the following characterization of the Runge approximation property, see \cite[Theorem 6]{Kalmes21} and \cite[Theorem 26.1]{Treves1967-2}.

\begin{proposition}\label{abstract-Runge}
	Let $P\in\C[\xi_1,\ldots,\xi_d] \backslash \{0\}$ and let $X,Y \subseteq\R^d$ be open with $Y \subseteq X$ such that $X$ is $P$-convex for supports. Then, the following statements are equivalent:
		\begin{itemize}
			\item[(i)] $Y$ is $P$-convex for supports and the restriction map $$
			r_{\mathscr{E}}^P:\mathscr{E}_P(X)\rightarrow\mathscr{E}_P(Y), \, f \mapsto f_{\mid Y}
			$$ has dense range.
		\item[(ii)] $Y$ is $P$-convex for supports and the restriction map $$
		r_{\mathscr{D}'}^P:\mathscr{D}'_P(X)\rightarrow\mathscr{D}'_P(Y), \, f \mapsto f_{\mid Y}
		$$ has dense range.
		\item[(iii)] For every $\varphi\in\mathscr{D}(X)$ with $\supp\check{P}(D)\varphi\subseteq Y$ it holds that $\supp\varphi\subseteq Y$.
	\end{itemize}
\end{proposition}

We start with a  useful geometrical lemma.

\begin{lemma}\label{lem: entailing minimum principle}
	Let $Y \subseteq X\subseteq\R^d$ be open. Let $F\subseteq\R^d$ be  a non-empty closed set such that $d_X$ satisfies the minimum principle in $x+F$ for every $x\in\R^d$. Suppose that there is no $x\in\R^d$ such that $X$ contains a compact connected component of $(\R^d\backslash Y)\cap (x+F)$. Then, $d_Y$ satisfies the minimum principle in $x+F$ for every $x\in\R^d$.
\end{lemma}

\begin{proof}
	We may assume $Y\neq X$. We argue by contradiction, so assume that there is $x\in \R^d$ and  $K\subseteq Y\cap(x+F)$ compact such that
$$
		\min_{z\in K}d_Y(z)<\min_{z\in\partial_{x+F}K}d_Y(z).
$$
	Choose $z_0\in K$ and $y_0\in\R^d\backslash Y$ with $|z_0-y_0|=\min_{z\in K}d_Y(x)$.
Hence,
	\begin{equation}\label{eq: violation minimum principle 2}
		|z_0-y_0|<\inf\{|y-z| \,; \, y\in\R^d\backslash Y, z\in \partial_{x+F}K\}.
	\end{equation}
	This implies that $\partial_{y_0-z_0+x+F}(y_0-z_0+K) = y_0-z_0+\partial_{x+F}K \subseteq Y$. Since $y_0 \in y_0-z_0+x+F$, we obtain that the connected component $C_{y_0}$ of $(\R^d\backslash Y)\cap (y_0-z_0+x+F)$ containing $y_0$ satisfies $C_{y_0}\subseteq y_0-z_0+K$. In particular, $C_{y_0}$ is compact. If $X=\R^d$, this gives the desired contradiction. 	If $X\neq\R^d$, we conclude that $\emptyset\neq\partial X\cap C_{y_0}$. Set
	$$t_0=\inf\{t\in[0,\infty) \, ; \,t(y_0-z_0)+K\cap\partial X\neq \emptyset\}.$$
 We have that $t_0 > 0$ because $K \subseteq X$ is compact, while  $C_{y_0}\subseteq y_0-z_0+K$ and $\emptyset\neq\partial X\cap C_{y_0}$ imply that $t_0 \leq 1$. Note that $t(y_0-z_0)+K\subseteq X$ for all $t \in (0,t_0)$. Since $t_0(y_0-z_0)+K\cap\partial X\neq \emptyset$, it holds that
	\begin{equation}\label{eq: contradiction minimum principle 1}
	 d\left(t(y_0-z_0)+K, \R^d\backslash X\right)\leq (t_0-t)|y_0-z_0|, \qquad 	\forall\,t\in (0,t_0).
	\end{equation}
	Set 
	$$\delta=\inf\{|y-z|\, ; \, y\in\R^d\backslash Y, z\in \partial_{x+F}K\}-|z_0-y_0|,$$
	so that $\delta>0$ by \eqref{eq: violation minimum principle 2}. For all $y\in \R^d\backslash X\subseteq\R^d\backslash Y$, $z\in \partial_{x+F}K$ and $t\in (0,t_0)$ we have that
	$$|y-\left(t(y_0-z_0)+z\right)|\geq |y-z|-t|z_0-y_0|\geq|y-z|-|z_0-y_0|\geq\delta.$$
	Hence, for $t\in (0,t_0)$ it follows that
	$$\min_{z'\in\partial_{t(y_0-z_0)+x+F}(t(y_0-z_0)+K)}d_X(z')=\inf\{|y-\left(t(y_0-z_0)+z\right)\,; \, y\in\R^d\backslash X, z\in \partial_{x+F}K\}\geq \delta.$$
	By evaluating \eqref{eq: contradiction minimum principle 1} for $t\in(0,t_0)$ with $(t_0-t)|y_0-z_0|<\delta/2$, we obtain that
	$$\min_{z'\in t(y_0-z_0)+K}d_X(z')<\min_{z'\in\partial_{t(y_0-z_0)+x+F}(t(y_0-z_0)+K)}d_X(z').$$
	Since $t(y_0-z_0)+K \subseteq X \cap (t(y_0-z_0)+ x+ F)$, $d_X$ does not satisfy the minimum principle in $t(y_0-z_0)+ x+ F$, a contradiction.
\end{proof}

We are ready to prove the main result of this section,
its proof is inspired by the one of  \cite[Theorem 1]{Kalmes21}. Given $X \subseteq \R^d$, we denote by $\partial_\infty X$  the boundary of $X$ in the one-point compactification of $\R^d$. Thus, $\partial_\infty X=\partial X$ if $X$ is bounded, and $\partial_\infty X=\partial X\cup\{\infty\}$ otherwise. 
\begin{theorem}\label{theo: Runge for subspace elliptic}
	Let $P\in\C[\xi_1,\ldots,\xi_d]$ be a polynomial which acts along a subspace $W$ of $\R^d$ and is elliptic on $W$. Let $X,Y \subseteq\R^d$ be open with $Y \subseteq X$ such that $X$ is $P$-convex for supports. Suppose that there is no $x\in\R^d$ such that $X$ contains a compact connected component of $(\R^d\backslash Y)\cap (x+W)$. Then, $Y$ is $P$-convex for supports and both restriction maps
	$r_\mathscr{E}^P: \mathscr{E}_P(X)\rightarrow\mathscr{E}_P(Y)$ and $r_{\mathscr{D}'}^P:\mathscr{D}'_P(X)\rightarrow\mathscr{D}'_P(Y)$
	have dense range.
\end{theorem}

\begin{proof}
	Since $X$ is $P$-convex for supports, Theorem \ref{theo: cases of Omega}(a) implies that $d_X$ satisfies the minimum principle in $x+W$ for every $x\in\R^d$. Hence, by the hypothesis  on the geometry of $X$ and $Y$, and Lemma \ref{lem: entailing minimum principle}, $d_Y$ satisfies the minimum principle in $x+W$ for every $x\in\R^d$ as well. Another application of Theorem \ref{theo: cases of Omega}(a) yields that  $Y$ is $P$-convex for supports.

By Proposition \ref{abstract-Runge}, it suffices to show that for every $\varphi\in\mathscr{D}(X)$ with $\supp\check{P}(D)\varphi\subseteq Y$ it holds that $\supp\varphi\subseteq Y$. 
Set $K=\supp\check{P}(D)\varphi\subseteq Y$. We claim that 
	$$\left\{y\in Y\, ; \, \,d_Y(y)<d(K,\R^d\backslash Y)\right\} \cap \supp \varphi = \emptyset.$$
Once we have proved this, $ch(\supp\varphi)=ch(K)$ (theorem of supports) implies that
	$$\supp\varphi\subseteq\left[\{y\in Y\, ; \, \,d_Y(y)\geq d(K,\R^d\backslash Y)\}\cap ch(K)\right]\cup\left[(X\backslash Y)\cap\supp\varphi\right].$$
Set $K_1=\{y\in Y\, ; \, \,d_Y(y)\geq d(K,\R^d\backslash Y)\}\cap ch(K)$ and $K_2=(X\backslash Y)\cap\supp\varphi$. Then,  $\supp \varphi$ is contained in the union of the disjoint compact sets $K_1\subseteq Y$ and $K_2\subseteq X\backslash Y$. Hence, $\varphi=\varphi_1+\varphi_2$ with $\varphi_j\in\mathscr{D}(K_j), j=1,2$, and thus $\check{P}(D)\varphi_2= 0$. The injectivity of $\check{P}(D)$ on $\mathscr{D}(\R^d)$ yields that $\varphi=\varphi_1$, whence  $\supp\varphi\subseteq Y$.
	
We now show the claim. Let $x \in Y$ be such that  $d_Y(x)<d(K,\R^d\backslash Y)$. We need to show that $x \notin \supp \varphi$.  We claim that there exists a continuous, piecewise affine curve $\gamma:[0,\infty)\rightarrow \left(X\backslash K\right)\cap(x+W)$ with $\gamma(0)=x$ and  $\lim_{t\rightarrow\infty}d(\gamma(t),\partial_\infty X)=0$.  Let us assume, for the moment, that we have shown that  such a curve $\gamma$ exists. Since $\lim_{t\rightarrow\infty}d(\gamma(t),\partial_\infty X)=0$, there exists $T > 0$ such that $\gamma(T)\notin\supp\varphi$. Let $t_0=0<t_1<\ldots<t_n=T$ be a partition of $[0,T]$ such that $\gamma$ is affine in $[t_j,t_{j+1}]$ for every $j=0,\ldots,n-1$. Let $\varepsilon>0$ be such that $B(\gamma(T),\varepsilon) \cap \supp\varphi = \emptyset$ and $\left(\gamma([0,T])+B(0,\varepsilon)\right) \cap K= \emptyset$. Consider the open convex sets $X_2=\gamma([t_{n-1},T])+B(0,\varepsilon)$ and $X_1=B(\gamma(T),\varepsilon)\subseteq X_2$ ($X_2$ is convex because $\gamma $ is affine on $[t_{n-1},T]$). Then, $\varphi$ vanishes in $X_1$ and  $\check{P}(D)\varphi$ vanishes in $X_2$. Note that every characteristic hyperplane for $P(D)$ is of the form $\{y\in\R^d \, ; \, \langle N, y \rangle =c \}$, with $c \in\R$ and $N \in W^\perp \backslash \{0\}$. Hence, each characteristic hyperplane that intersects $X_2$ already intersects $X_1$. It follows from \cite[Theorem 8.6.8]{HoermanderPDO1} that $\varphi$ vanishes in $X_2$. Iteration of this argument yields that $\varphi$ vanishes in $\gamma([t_j,t_{j+1}])+B(0,\varepsilon)$ for every $j=0,\ldots, n-1$. In particular, $\varphi$ vanishes on $B(x, \varepsilon) = B(\gamma(t_0), \varepsilon)$  and thus $x \notin \supp \varphi$.
	
We now show the existence of a  curve $\gamma$ with the above properties. Let $C$ be the  connected component of $(X\backslash K)\cap (x+W)$ that contains $x$. Then, $C$ is pathwise connected. If $C$ is unbounded, there is a continuous, piecewise affine curve $\gamma:[0,\infty)\rightarrow C$ with $\gamma(0)=x$ and $\lim_{t\rightarrow\infty}|\gamma(t)|=\infty$. In particular, $\gamma$ has the desired properties. If $C$ is bounded, as $Y$ is $P$-convex for supports,  \cite[Theorem 10.8.5]{HoermanderPDO2} and \cite[Lemma 4]{Kalmes19} (see also the proof of \cite[Theorem 1]{Kalmes21}) imply that there exists a continuous, piecewise affine curve $\alpha:[0,\infty)\rightarrow(Y\backslash K)\cap(x+W)$ with $\alpha(0)=x$ and $\liminf_{t\rightarrow\infty}d(\alpha(t),\partial_\infty Y)=0$. Note that $\alpha([0,\infty))\subseteq C$. Since $C$ is bounded, we obtain that  $\liminf_{t\rightarrow\infty}d(\alpha(t),\partial Y)=0$. Hence, there are $\xi\in \partial Y \cap(x+W)$, $T>0$, and $\varepsilon>0$ such that $B(\xi,\varepsilon)\cap K=\emptyset$ and $\alpha(T)\in B(\xi,\varepsilon)$.  If $\xi\in\partial X$, we consider the  line  $\beta:[T,\infty)\rightarrow B(\xi, \varepsilon) \cap(x+W)$  with $\beta(T)=\alpha(T)$ and $\lim_{t\rightarrow\infty}\beta(t)=\xi$. The concatenation of the curves $\alpha_{|[0,T]}$ and $\beta$  defines a curve $\gamma$ with the desired properties. If $\xi\notin\partial X$, we may assume that $B(\xi,\varepsilon)\subseteq X$. Consider the line $\beta_1:[T,T+1]\rightarrow B(\xi, \varepsilon) \cap(x+W)$  with $\beta_1(T)=\alpha(T)$ and $\beta_1(T+1)=\xi$. Let $C_\xi$ and $D_\xi$ be the connected components of $(\R^d\backslash Y)\cap(x+W)$ and $(\R^d\backslash K)\cap(x+W)$, respectively, that contain $\xi$. Then, $C_\xi\subseteq D_\xi$ and $D_\xi$ is pathwise connected.  If $C_\xi$ is bounded,  by the hypothesis on the geometry of $X$ and $Y$, $C_\xi$ contains some $\eta\in\partial X$. Hence, there exists a continuous, piecewise affine curve $\beta_2:[T+1,\infty)\rightarrow D_\xi$ with $\beta_2(T+1)=\xi$ and $\lim_{t\rightarrow\infty}\beta_2(t)=\eta$. The concatenation of the curves $\alpha_{|[0,T]}$, $\beta_1$, and $\beta_2$ defines a curve $\gamma$ with the desired properties. If $C_\xi$ is unbounded,  there exists a continuous, piecewise affine curve $\beta_2:[T+1,\infty)\rightarrow D_\xi$ with $\beta_2(T+1)=\xi$ and $\lim_{t\rightarrow\infty}|\beta_2(t)|=\infty$. If the range of $\beta_2$ does not intersect $\partial X$, the concatenation of the curves $\alpha_{|[0,T]}$, $\beta_1$, and $\beta_2$ defines a curve $\gamma$ with the desired properties. If the range of $\beta_2$ intersects $\partial X$, we consider $S=\inf\{t\in[T+1,\infty)\,;\,\beta_2(t)\in\partial X\} < \infty$.  By a suitable reparametrization of $\beta_{2|[T+1,S)}$, we obtain a continuous, piecewise affine curve $\beta_3:[T+1,\infty)\rightarrow D_\xi$ with $\beta_3(T+1)=\xi$ and $\lim_{t\rightarrow\infty}d(\beta_3(t),\partial X)=0$. Again, the desired curve $\gamma$ is obtained by concatenating the curves $\alpha_{|[0,T]}$, $\beta_1$, and $\beta_3$.
\end{proof}

\begin{remark}\label{remark-LM}
	For $W=\R^d$ Theorem \ref{theo: Runge for subspace elliptic} is of course the Lax-Malgrange theorem for elliptic operators  \cite{Lax, Malgrange}. The condition on $X$ and $Y$ becomes that X contains no compact connected component of $\R^d\backslash Y$. In this case, the condition is also necessary for the restriction maps to have dense range.
\end{remark}

We now use Theorem \ref{theo: Runge for subspace elliptic} to give an approximation result for differential operators that factor into first order operators. We need the following result.

\begin{lemma}\label{prop: composition of differential operators}
	Let $P_1,P_2\in\C[\xi_1,\ldots,\xi_d]$ and set $P=P_1P_2$.
	\begin{itemize}
		\item[(a)] An open set $X\subseteq\R^d$ is $P$-convex for supports if and only if $X$ is $P_1$-convex for supports as well as $P_2$-convex for supports. An analogous statement holds for convexity for singular supports.
		\item[(b)] Assume that $Y \subseteq X\subseteq\R^d$ are open sets that are $P$-convex for supports. Then, $r_\mathscr{E}^P$ has dense range if and only if both $r_\mathscr{E}^{P_1}$ and $r_\mathscr{E}^{P_2}$ have dense range. An analogous statement holds for the distributional restriction maps.
	\end{itemize}
\end{lemma}

\begin{proof}
	(a) is  obvious, while (b) follows from Proposition \ref{abstract-Runge}.
\end{proof}

\begin{theorem}\label{theo: Runge for affine factors}
 Let $P \in\C[\xi_1,\ldots,\xi_d]$  be such that $P(\xi)=\alpha\prod_{j=1}^l\left(\langle N_j,\xi\rangle+c_j\right)$, where $\alpha\in\C\backslash\{0\}$, $N_j\in \C^d\backslash\{0\}$, $c_j\in\C$, $j=1,\ldots,l$. 
 Let $X,Y \subseteq\R^d$ be open with $Y \subseteq X$ such that $X$ is $P$-convex for supports. Consider the following statements:
	\begin{itemize}
			\item[(i)] $Y$ is $P$-convex for supports and the restriction map $r_{\mathscr{E}}^P:\mathscr{E}_P(X)\rightarrow\mathscr{E}_P(Y)$ has dense range.
		\item[(ii)] $Y$ is $P$-convex for supports and the restriction map $r_{\mathscr{D}'}^P:\mathscr{D}'_P(X)\rightarrow\mathscr{D}'_P(Y)$ has dense range.
		\item[(iii)] There is no $x\in\R^d$ such that $X$ contains a compact connected component of any of the sets $(\R^d\backslash Y)\cap\left(x+\mbox{span}\{\mbox{Re}N_j, \mbox{Im}N_j\}\right)$, $j=1,\ldots, l$.
	\end{itemize}
	Then, $(iii) \Rightarrow(i) \Leftrightarrow (ii)$. If $d=2$, the three statements are equivalent.
\end{theorem}

\begin{proof} Set $P_j(\xi)=\langle N_j,\xi\rangle+c_j$ and $W_j=\mbox{span}\{\mbox{Re}N_j, \mbox{Im}N_j\}$, $j=1,\ldots, l$. Note that $P_j(D)$ acts along $W_j$ and is elliptic on $W_j$ for all  $j=1,\ldots, l$. The equivalence $(i) \Leftrightarrow (ii)$ follows  from  Proposition \ref{abstract-Runge}, while  the implication $(iii) \Rightarrow (i)$ is a consequence of Theorem \ref{theo: Runge for subspace elliptic} and Lemma \ref{prop: composition of differential operators}. Now assume that $d=2$. We will show that (ii) implies (iii). By Lemma \ref{prop: composition of differential operators}(b), the restriction map $r_{\mathscr{D}'}^{P_j}$ has dense range for all $j=1,\ldots,l$. If $\mbox{span}\{\mbox{Re}N_j, \mbox{Im}N_j\}=\R^2$, $P_j(D)$ is elliptic and the Lax-Malgrange theorem (see also Remark \ref{remark-LM}) implies that there is no $x\in\R^2$ such that $X$ contains a compact connected component of
	$$
	 \R^2\backslash Y = \left(\R^2\backslash Y\right)\cap\left(x+\mbox{span}\{\mbox{Re}N_j, \mbox{Im}N_j\}\right).
	$$ 
	If $\mbox{span}\{\mbox{Re}N_j, \mbox{Im}N_j\}\neq\R^2$, by making a change of variables, we may assume that $\mbox{span}\{\mbox{Re}N_j, \mbox{Im}N_j\}=\mbox{span}\{e_2\}$ with $e_2=(0,1)\in\R^2$. Then, $P_j(\xi)=\beta+\gamma\xi_2$ for suitable $\beta,\gamma\in\C$, $\gamma\neq 0$, and thus $\operatorname{span}\{e_2\}$ is  a characteristic line for $P_j(D)$ but $\operatorname{span}\{e_1\}$ is not.  Since $r_{\mathscr{D}'}^{P_j}$ has dense range, \cite[Theorem 2]{Kalmes21} implies that there is no $x\in\R^2$ such that $X$ contains a compact connected component of
	$$\left(\R^2\backslash Y\right)\cap\left(x+\operatorname{span}\{e_2\}\right).$$
\end{proof}

Note that Theorem \ref{theo: Runge for affine factors} is applicable to homogeneous polynomials in two variables, in particular to the  wave operator in one spatial variable. In the next result we further study  the Runge approximation property for zero solutions of  this operator.

\begin{theorem}\label{theo: Runge for wave operator}
	Let $P(D)=\frac{\partial^2}{\partial x_1^2}-\frac{\partial^2}{\partial x_2^2}$ be the wave operator  in one spatial variable and let $Y\subseteq\R^2$ be open and $P$-convex for supports. The following statements are equivalent:
		\begin{itemize} 
			\item[(i)] $r_{\mathscr{E}}^P:\mathscr{E}_P(\R^2)\rightarrow\mathscr{E}_P(Y)$ has dense range.
			\item[(ii)] There is an open connected set $X \subseteq \R^2$ with $Y \subseteq X$  such that $X$ is $P$-convex for supports and the restriction map $r_{\mathscr{E}}^P:\mathscr{E}_P(X)\rightarrow\mathscr{E}_P(Y)$ has dense range.
			\item[(iii)] For all   $r\in\N_0\cup\{\infty\}$ the following extension property holds:  For each $f \in C^r(Y)\cap\mathscr{D}'_P(Y)$  and $K \subseteq Y$ compact  there is  $g\in C^r(\R^2)\cap\mathscr{D}'_P(\R^2)$ with compact support such that  $g=f$ on a neighborhood of $K$.
		\end{itemize}
\end{theorem}

\begin{proof}
The implications $(i) \Rightarrow (ii)$ and $(iii) \Rightarrow (i)$ are obvious. We now show $(ii) \Rightarrow (iii)$. By making the change of variables $(y_1,y_2) = (\xi_2-\xi_1,\xi_1+\xi_2)$, it suffices to show the following statement for the operator $Q(D_y)=\frac{\partial^2}{\partial y_1\partial y_2}$ and an open set $\tilde{Y}\subseteq\R^2$ that is $Q$-convex for supports: 

\emph{Let $\tilde{X} \subseteq \R^2$ be open and connected with $\tilde Y \subseteq \tilde X$  such that $\tilde X$ is $Q$-convex for supports and   the restriction map $r_{\mathscr{E}}^Q:\mathscr{E}_Q(\tilde X)\rightarrow\mathscr{E}_Q(\tilde Y)$ has dense range. Then, for all   $r\in\N_0\cup\{\infty\}$, $f \in C^r(\tilde{Y})\cap\mathscr{D}'_Q(\tilde Y)$,  and $K \subseteq \tilde Y$ compact  there is  $g\in C^r(\R^2)\cap\mathscr{D}'_Q(\R^2)$ with compact support such that $g=f$ on a neighborhood of $K$}.  

By Theorem \ref{theo: Runge for affine factors},  there is no $y\in\R$ such that $\tilde{X}$ contains a compact connected component of one of the sets  $(\R^2\backslash \tilde{Y})\cap\{(s,y) \,; \, s\in\R\}$ or $(\R^2\backslash \tilde{Y})\cap\{(y,s)\, ; \, s\in\R\}$. Let $\pi_k:\R^2\rightarrow\R, (y_1,y_2)\mapsto y_k$, $k=1,2$. For a connected component $C$ of $\tilde{Y}$ we set $R(C)=\pi_1(C)\times\pi_2(C)$.
	
	\textsc{Auxiliary Claim 1:} \emph{If $C_1$ and $C_2$ are disjoint connected components of $\tilde{Y}$, then $R(C_1)$ and $R(C_2)$ are disjoint as well.}
	
	\emph{Proof of auxiliary claim 1}. Suppose that $(y_1,y_2)\in R(C_1)\cap R(C_2)$. We find $z_k\in\R$ such that $(z_k,y_2)\in C_k$, $k = 1,2$. Since $C_1$ and $C_2$ are disjoint connected components of $\tilde Y$, there is $\lambda\in (0,1)$ such that $(\lambda z_1+(1-\lambda)z_2,y_2)\in\R^2\backslash \tilde{Y}$. The geometric relation of $\tilde{X}$ and $\tilde{Y}$ implies that there is $\mu\in (0,1)$ such that $(\mu z_1+(1-\mu)z_2,y_2)\notin\tilde{X}$. Hence, the intersection of $\tilde{X}$ with the line $\{(s,y_2)\, ; \,s\in\R\}$ has at least two non-empty connected components (as it contains $(z_k,y_2)$, $k=1,2$). As this line is characteristic for $Q$ and $\tilde{X}$ is connected and $Q$-convex for supports, the desired contradiction follows from \cite[Theorem 10.8.3]{HoermanderPDO2}. 
	
	Let   $r\in\N_0\cup\{\infty\}$, $f \in C^r(\tilde{Y})\cap\mathscr{D}'_Q(\tilde Y)$,  and $K \subseteq \tilde Y$ be arbitrary.
	
	\textsc{Auxiliary Claim 2:} \emph{For each connected component $C$ of $\tilde Y$ there are  $g_{k,C}\in C^r(\pi_k(C))$, $k=1,2$, with $f(y_1,y_2)=g_{1,C}(y_1)+g_{2,C}(y_2)$, $(y_1,y_2)\in C$.}
		
	\emph{Proof of auxiliary claim 2}. It is clear that for any open rectangle $R\subseteq C$ there are $g_{k,R}\in C^r(\pi_k(R))$, $k=1,2$, with $f(y_1,y_2)=g_{1,R}(y_1)+g_{2,R}(y_2)$, $(y_1,y_2)\in R$.  For two such rectangles $R_1$ and $R_2$ that are overlapping, there is $\alpha\in\C$ such that
	$$
	g_{1,R_1}(y_1)=g_{1,R_2}(y_1)+\alpha, \qquad g_{2,R_1}(y_2)=g_{2,R_2}(y_2)-\alpha,
	$$
	for all $y_1\in\pi_1(R_1\cap R_2)$, $y_2\in\pi_2(R_1\cap R_2)$. This implies that $g'_{k,R_1}=g'_{k,R_2}$ on $\pi_k(R_1\cap R_2)$, $k=1,2$. It follows that there are $h_k\in\mathscr{D}'(\pi_k(C))$, $k=1,2$, that satisfy $h_{k|\pi_k(R)}=g'_{k,R}$ on every open rectangle $R\subseteq C$. Since $\pi_j(C)$ is an interval, there are $\tilde{g}_k\in C^r(\pi_k(C))$, $k=1,2$, with $h_k=\tilde{g}'_k$ (for $r \geq 1$ this is clear, while the case $r =0$  follows from   \cite[Exercise 3.1.1 and Corollary 3.1.5]{HoermanderPDO1}). Hence,
	$$\tilde{g}:\pi_1(C)\times\pi_2(C)\rightarrow\C, (y_1,y_2)\mapsto\tilde{g}_1(y_1)+\tilde{g}_2(y_2)$$
	satisfies $\frac{\partial \tilde{g}}{\partial y_k}= \frac{\partial f}{\partial y_k}$, $k = 1,2$ on $C$. Now, if $\gamma:(0,1)\rightarrow C$ is an arbitrary $C^1$-curve with $\overline{\gamma(0,1)}\subseteq C$, we have that   $((f-\tilde{g})\circ\gamma)' = 0$ in  $\mathscr{D}'(0,1)$   (for $r \geq 1$ this is clear, while the case $r =0$ requires a standard mollification argument). By \cite[Theorem 3.1.4]{HoermanderPDO1}, we obtain that  $f-\tilde{g}$  is constant along $C^1$-curves $\gamma:(0,1)\rightarrow C$ with $\overline{\gamma(0,1)}\subseteq C$. Since $C$ is pathwise connected (as $\tilde{Y}$ is open), we conclude that $f-\tilde{g}$ is constant on $C$. This implies the claim.

	We are ready to show the result.  Set $\mathscr{C}=\{C \,; \,C\mbox{ connected component of }\tilde{Y}\}$. For each $C\in\mathscr{C}$ we choose $\varphi_{k,C}\in\mathscr{D}(\R)$ with $\supp\varphi_{k,C}\subseteq\pi_k(C)$ and $\varphi_{k,C}=1$ on a neighborhood of $\pi_k(K\cap C)$, $k=1,2$. Define

	$$g:\R^2\rightarrow\C, \, g(y_1,y_2)=\begin{cases}
		\varphi_{1,C}(y_1)g_{1,C}(y_1)+\varphi_{2,C}(y_2)g_{2,C}(y_2),&\mbox{if }(y_1,y_2)\in R(C)\\
		0,&\mbox{if }(y_1,y_2)\notin \bigcup_{C\in\mathscr{C}}R(C),
	\end{cases}$$
	where $g_{k,C}$, $k=1,2$, are the functions from the auxiliary claim 2. The function $g$ is well-defined by the auxiliary claim 1. Furthermore, it is clear that it satisfies all requirements.
	\end{proof}

	\begin{remark} Let $P(D)=\frac{\partial^2}{\partial x_1^2}-\frac{\partial^2}{\partial x_2^2}$. Theorem \ref{theo: Runge for wave operator} implies that the restriction map $r_{\mathscr{E}}^P:\mathscr{E}_P(\R^2)\rightarrow\mathscr{E}_P(Y)$ has dense range for every connected open subset $Y$ of $\R^2$ that is $P$-convex for supports (choose $X = Y$ in condition (ii)).
\end{remark}

\begin{remark}\label{rem: Runge for one dimensional wave operator}  We use the same notation as in the proof of Theorem \ref{theo: Runge for wave operator} and consider the case $r=\infty$. For $\tau > 0$ we set $f_\tau(y_1,y_2)=e^{i\tau y_1}\in\mathscr{E}_Q(\tilde{Y})$. For $C\in\mathscr{C}$  the  components of the decomposition of $f_\tau$ in the auxiliary claim 2  are given by  $f_{1,C}(y_1)=e^{i\tau y_1}$ and $f_{2,C}(y_2)=0$ . Hence, for the corresponding extension  $g_\tau$ it holds that for all $r \in \N_0$ and $L \subseteq \R^2$ compact with $K \subseteq L$ there is $C > 0$ such that
$$
\tau^r \leq \| g_\tau \|_{L,r} \leq C \tau^r, \qquad \forall \tau > 0.
$$
This shows that  the construction used in the proof of Theorem \ref{theo: Runge for wave operator} cannot be employed to show quantitative Runge type approximation results for the  wave operator in one spatial variable (such as Theorem \ref{theo: quantitative Runge for wave operator}).
	\end{remark}
Finally, we give a slight improvement of \cite[Theorem 1]{Kalmes21} for semi-elliptic operators.

\begin{theorem}\label{theo: qualitative Runge for single characteristic direction}
	Let $P\in\C[\xi_1,\ldots,\xi_d]$ be semi-elliptic  with principal part $P_m$ such that $\{P_m=0\}$ is one-dimensional. Let $X,Y\subseteq\R^d$ be open with $Y \subseteq X$ such that $X$ is  $P$-convex for supports. Suppose that there is no $x\in\R^d$ such that $X$ contains a compact connected component of $(\R^d\backslash Y)\cap(x+\{P_m=0\}^\perp)$. Then, $Y$ is $P$-convex for supports and the restriction map $r_\mathscr{E}^P: \mathscr{E}_P(X)\rightarrow\mathscr{E}_P(Y)$ has dense range.
\end{theorem}

\begin{proof}
	By  \cite[Theorem 1]{Kalmes21}, it suffices to show that  $Y$ is $P$-convex for supports. Since $X$ is $P$-convex for supports,  Theorem \ref{theo: cases of Omega}(b) implies that $d_X$ satisfies the minimum principle in every characteristic hyperplane, that is, in $x+\{P_m=0\}^\perp$ for every $x\in\R^d$. Lemma \ref{lem: entailing minimum principle} yields that $d_Y$ satisfies the minimum principle in every characteristic hyperplane as well. Hence, another application of Theorem \ref{theo: cases of Omega}(b) shows that  $Y$ is $P$-convex for supports. 
\end{proof}

\section{An auxiliary quantitative approximation result}\label{sec: technical}

In this section, we show an internal quantitative approximation result for  $\mathscr{E}_P(X)$. Namely, we show condition $(\Omega)$  for $\mathscr{E}_P(X)$ (see \eqref{def-Omega}) with control on the compact $L= L(K,k)$ and the number $l = l(K,k)$ occurring in it.  We will combine this result with the qualitative Runge type approximation theorems from Section \ref{sec: qualitative Runge} to obtain quantitative approximation theorems in the next section (cf.\ the explanation at the  end of Section \ref{sec: prelim}).

We start with a definition that will play a central role in the rest of our work.

\begin{definition}\label{def: augmenedtly locating}
	Let $P\in\C[\xi_1,\ldots,\xi_d]$, let $X\subseteq\R^d$ be open and let $K \subseteq L \subseteq X$ be compact sets. We say that the pair $(K,L)$ is \emph{augmentedly $P$-locating for $X$} if for every interval $[a,b]\subseteq\R$ it holds that
	\begin{equation}\label{eq: augmentedly locating supports}
	\supp\check{P}^+(D)f\subseteq K\times[a,b] \, \Rightarrow \,\supp f\subseteq L\times[a,b], \qquad 	\forall\, f \in\mathscr{E}'(X\times\R),
	\end{equation}
and 
	\begin{equation}\label{eq: augmentedly locating singular supports}
	\singsupp\check{P}^+(D) f \subseteq K\times[a,b] \, \Rightarrow \, \singsupp f \subseteq L\times[a,b],  \qquad	\forall\, f \in\mathscr{E}'(X\times\R).
	\end{equation}
A compact set $K \subseteq  X$ is said to be augmentedly $P$-locating for $X$ if 
 the pair $(K,K)$ is so.
\end{definition}

In Proposition \ref{cor: explicit Omega} below we show how this notion may be used to formulate a general internal quantitative approximation result for $\mathscr{E}_P(X)$. First, we give sufficient geometric conditions ensuring that a pair of compact sets $(K,L)$  is augmentedly $P$-locating for an open set $X \subseteq \R^d$. The next result is nothing else but a reformulation of the theorems of supports and singular supports.

 \begin{lemma}\label{prop: location of sing/supp-convex}
 	Let $P \in\C[\xi_1,\ldots, \xi_d] \backslash \{0\}$. Every convex set  $K \subseteq \R^d$ is augmentedly $P$-locating for $\R^d$.
 \end{lemma}
Although we will not explicitly use it, we would like to point out the following generalization of Lemma \ref{prop: location of sing/supp-convex}.
\begin{lemma}\label{prop: location of sing/supp-simple}
	Let $P \in\C[\xi_1,\ldots,\xi_d]$ and let $X\subseteq\R^d$ be open such that $P(D)^+: \mathscr{D}'(X\times\R) \rightarrow \mathscr{D}'(X\times\R)$ is surjective. Let $K \subseteq X$ be compact and define 
	$$\hat{K}_X=ch(K)\cap\{x\in X \,; \,d_X(x)\geq d(K,\R^d\backslash X)\}.$$
	Then, $(K,\hat{K}_X)$ is augmentedly $P$-locating for $X$.
\end{lemma}
\begin{proof} Set $\delta =d(K,\R^d\backslash X)$. 
		 Let $[a,b] \subseteq \R$ be an arbitrary interval and let  $f \in\mathscr{E}'(X\times\R)$ be such that $\operatorname{(sing)\, supp}\check{P}^+(D)f \subseteq K\times[a,b]$. Since  $X\times\R$ is $P^+$-convex for supports and singular supports, \cite[Theorem 10.6.3 and Theorem 10.7.3]{HoermanderPDO2} imply that
		 \begin{eqnarray*}
		\delta= d(K\times[a,b],\R^d \backslash X \times\R) &\leq& d(\operatorname{(sing)\, supp} \check{P}(D)f ,\R^d \backslash X \times\R) \\
		&=&d(\operatorname{(sing)\, supp} f,\R^d \backslash X \times\R).
		 \end{eqnarray*}
By	the theorems of supports and singular supports, we have that
$$ch(\operatorname{(sing)\, supp} f)=ch(\operatorname{(sing)\, supp}\check{P}^+(D)f) \subseteq ch(K) \times [a,b].$$
Hence,
	\begin{eqnarray*}
		\operatorname{(sing)\, supp}\supp f&\subseteq& \{(x,t)\in ch(K)\times[a,b] \, ; \,\delta\leq d\left((x,t),\R^d \backslash X \times\R\right)=d_X(x)\} \\
		&=&\hat{K}_X\times[a,b].
	\end{eqnarray*}
\end{proof}

Next, we consider subspace elliptic  operators and semi-elliptic operators with a single characteristic direction.
 
\begin{lemma}\label{lem: location of sing/supp-subspace}
	Let $P \in\C[\xi_1,\ldots,\xi_d]$ be  a polynomial which acts along a subspace $W$ of $\R^d$ and is elliptic on $W$. Let $X \subseteq \R^d$ be open and let $K \subseteq X$ be compact. Suppose that there is no $x\in\R^d$ such that  $X$ contains a bounded connected component of $(\R^d\backslash K)\cap (x+W)$. Then, $K$ is augmentedly $P$-locating for $X$.
\end{lemma}	

\begin{proof}Let $[a,b] \subseteq \R$ be an arbitrary interval.
The proofs of \eqref{eq: augmentedly locating supports} and \eqref{eq: augmentedly locating singular supports} (with $L = K$) are very similar. Therefore, we first show \eqref{eq: augmentedly locating singular supports} in detail, and sketch the proof of \eqref{eq: augmentedly locating supports} afterwards.  Let  $f \in\mathscr{E}'(X\times\R)$ be such that $\singsupp \check{P}^+(D)f \subseteq K\times[a,b]$. We need to show that  $\singsupp f \subseteq K\times[a,b]$. By the theorem of singular supports, $\singsupp f \subseteq ch(K)\times[a,b]$. Hence, it suffices to show that for  every  $x=(x',x_{d+1})\in(\R^d\backslash K) \times[a,b]$ it holds that $x \notin \singsupp f$. To this end, we use the same technique as in the proof of Theorem \ref{theo: Runge for subspace elliptic}.   Let $C$ be the connected component of $(\R^d\backslash K)\cap (x'+W)$ in the one-point compactification of $\R^d$ that contains $x'$. Then, $C$ is pathwise connected. Our hypothesis on the geometry of $X$ and $K$ implies that $C$ intersects $\partial_\infty X$. Hence, there exists a continuous, piecewise affine curve $\gamma:[0,\infty)\rightarrow C \cap \R^d \subseteq  (\R^d \backslash K) \cap (x' +W)$ with $\gamma(0)=x'$ and $\lim_{t\rightarrow\infty} \gamma(t) \in \partial_\infty X$. In particular, there exists $T >0$ such that $(\gamma(T),x_{d+1})\notin\supp f$.  Let $t_0=0<t_1<\ldots<t_n=T$ be a partition of $[0,T]$ such that $\gamma$ is affine in $[t_j,t_{j+1}]$ for every $j=0,\ldots,n-1$.
Let $\varepsilon>0$ be such that $B((\gamma(T),x_{d+1}),\varepsilon) \cap \supp f = \emptyset$  and $\left((\gamma ([0,T]),x_{d+1})+B(0,\varepsilon)\right) \cap \left(K \times [a,b]\right) = \emptyset$.  Consider the open convex sets $X_2 = (\gamma ([t_{n-1},T]),x_{d+1})+B(0,\varepsilon)$ and $X_1=B((\gamma (T),x_{d+1}), \varepsilon) \subseteq X_2$. Then, $f$ vanishes in  $X_1$ and $ \check{P}^+(D)f$ is smooth in $X_2$. Every hyperplane  of the form $ \{y\in\R^{d+1}\, ; \,\langle N, y \rangle =c \}$, with $c \in \R$ and $N \in W^\perp \times \R$, that intersects $X_2$ already intersects $X_1$. By \cite[Corollary 11.3.7]{HoermanderPDO2} and  \cite[Lemma 8]{Kalmes19}, we obtain that $f$ is smooth in $X_2$. Iteration of this argument yields that $f$ is smooth in $(\beta([t_j,t_{j+1}]),x_{d+1})+B(0,\varepsilon)$ for every $j=0,\ldots n-1$. In particular, $f$ is smooth in $B(x,\varepsilon) = B((\beta(0),x_{d+1}),\varepsilon)$ and thus $x\notin \singsupp f$.

The proof of  \eqref{eq: augmentedly locating supports}  is verbatim the same as the one of \eqref{eq: augmentedly locating singular supports}, one only has to refer to \cite[Theorem 8.6.8]{HoermanderPDO1} instead of \cite[Corollary 11.3.7]{HoermanderPDO2}  and  \cite[Lemma 8]{Kalmes19}, and note that a hyperplane  $\{y\in\R^{d+1}\, ; \,\langle N, y \rangle =c \}$, with $c \in \R$ and $N \in \R^{d+1}\backslash \{0\}$ is characteristic for $P^+(D)$ if and only if $N\in W^\perp\times\R$.
\end{proof}

\begin{lemma}\label{prop: location of sing/supp-onedim}
	Let $P\in\C[\xi_1,\ldots,\xi_d]$ be semi-elliptic with principal part $P_m$ such that $\{P_m=0\}$ is one-dimensional. Let $X \subseteq \R^d$ be open and let $K \subseteq X$ be compact. Suppose that there is no $x\in\R^d$ such that $X$ contains a  bounded connected component of $\left(\R^d\backslash K\right)\cap\left(x+\{P_m=0\}^\perp\right)$.  Then, $K$ is augmentedly $P$-locating for $X$.
\end{lemma}	

\begin{proof}
The proof  is verbatim the same as the one of Lemma \ref{lem: location of sing/supp-subspace}, one only has to replace $W$ by $\{P_m=0\}^\perp$, refer to  \cite[Theorem 1]{FrKa} instead of \cite[Lemma 8]{Kalmes19}, and note that a hyperplane  $\{y\in\R^{d+1}\, ; \,\langle N, y \rangle =c \}$, with $c \in \R$ and $N \in \R^{d+1}\backslash \{0\}$ is characteristic for $P^+(D)$ if and only if $N\in\{P_m=0\}\times\R$.	
\end{proof} 

Next, we turn our attention to approximation results. Given two open sets $X,Y \subseteq \R^d$, we write $Y \Subset X$ to indicate that $Y$ is a relatively compact open subset of $X$; we emphasize that we only use this notation for open sets. Let $X \Subset \R^d$. We define $\mathscr{D}'(\overline{X})$ as the dual of $\mathscr{D}(\overline{X})$ and set
$$
\mathscr{D}'_P(\overline{X}) = \{f \in \mathscr{D}'(\overline{X}) \, ; \, P(D) f = 0\}.
$$
Furthermore, we write
$$
B_{\overline{X},J}=\{f \in\mathscr{D}'_P(\overline{X})\, ; \, \| f \|^*_{\overline{X},J}\leq 1\}, \qquad J \in \N_0,
$$
with $\| f \|^*_{\overline{X},J}\ = \sup \{ |\langle f, \varphi \rangle | \, ; \, \varphi \in \mathscr{D}(\overline{X}),  \| \varphi \|_{\overline{X},J} \leq 1    \}$. 

The following result is the condition $(P\Omega)$ for $\mathscr{D}'_P(X)$ \cite{B-D2006,DeKa22} with control on the compact sets occurring in it. Its proof is inspired by \cite[Theorem 4.1 and Propositon 8.1]{B-D2006} and  \cite[Section 3.4.5]{Wengenroth}.

\begin{proposition}\label{lem: explicit POmega}
	Let $P\in\C[\xi_1,\ldots,\xi_d]$ and let $X \subseteq \R^d$ be open such that $P^+(D): \mathscr{D}'(X \times \R) \to \mathscr{D}'(X \times \R)$ is surjective. Let $X_1, X_2, X_3 \Subset X$ with $X_1 \subseteq X_2 \subseteq X_3$  such that $(\overline{X_1}, \overline{X_2})$ is augmentedly $P$-locating for $X$. Then,
	$$
	\forall M\in\N_0\, \exists K \in \N_0, s, C > 0 \,  \forall\,\varepsilon\in (0,1)\,:\, B_{2,M}\subseteq\frac{C}{\varepsilon^s} B_{3,K}+\varepsilon B_{1,0},
	$$
	where $B_{j,J} = B_{\overline{X}_j,J}$, $J\in \N_0$, $j=1,2,3$.
\end{proposition}

\begin{proof}
Fix  $0 < a_1 < a_2 < a_3$ arbitrary and set $Z_j =  \mathscr{D}_{P^+}'(\overline{X}_j\times[-a_j,a_j])$, $j=1,2,3$. For $J \in \N_0$ we define the Banach space 
\[ 
Z_{j,J} = \{ f  \in Z_j \,; \, \|f\|^*_{\overline{X}_j\times[-a_j,a_j], J} < \infty  \}
\]
and denote by $B_{j,J}^+$ the unit ball in $Z_{j,J}$. Note that $Z_j = \varinjlim_{J \to \infty} Z_{j,J}$. As shown in \cite[Section 3.4.5]{Wengenroth}, the facts that $X \times \R$ is $P^+$-convex for supports as well as for singular supports and $(\overline{X_1}, \overline{X_2})$ is augmentedly $P$-locating for $X$ imply that 
\[
	Z_2\subseteq Z_3+\{ f  \in L_2(\R^{d+1})\, ; \,\supp  f \subseteq\overline{X}_1\times[-a_1,a_1], P^+(D) f=0\}
		\subseteq Z_3+Z_{1,0}.
\]
Note that we did not distinguish notationally between distributions defined on $\overline{X_j}\times[-a_j,a_j]$, $j=2,3$, and their restriction to $\overline{X_1}\times[-a_1,a_1]$. We employ this convention throughout the proof. By using Grothendieck's factorization theorem, we conclude from the previous inclusion that 
	\begin{equation}\label{eq: decomposition}
	\forall M\in\N_0\, \exists K \in \N_0, C > 0 \,:\, B_{2,M}^+ \subseteq C(B_{3,K}^+ + B_{1,0}^+). 
 \end{equation}
	For two locally convex spaces $E$ and $F$ we denote by $L(E,F)$ the space of continuous linear operators from $E$ to $F$. We set  $W(A,B)=\{T\in L(E,F)\,; \,T(A)\subseteq B\}$, $A\subseteq E$, $B\subseteq F$. For $j  = 1,2,3$ we have that
	$$\Psi_j:\mathscr{D}'(\overline{X}_j\times[-a_j,a_j])\rightarrow L(\mathscr{D}[-a_j,a_j],\mathscr{D}'(\overline{X}_j)), $$
	with $\Psi_j(f):\mathscr{D}([-a_j,a_j])\rightarrow\mathscr{D}'(\overline{X}_j),\varphi\mapsto\left(\psi\mapsto \langle f,\psi\otimes\varphi \rangle\right)$, $f\in\mathscr{D}'(\overline{X}_j\times[-a_j,a_j])$, is an isomorphism (see the remark after \cite[Theorem 21.6.9]{Jarchow1981}). Note that 
\[
		\Psi_j\left(\mathscr{D}'_{P^+}(\overline{X}_j\times[-a_j,a_j])\right)=L\left(\mathscr{D}([-a_j,a_j]),\mathscr{D}'_P(\overline{X}_j)\right)
\]
	and, for $J \in \N_0$, 
$$
		\Psi_j\left(B_{j,J}^+\right)=W\left(\{\varphi\in\mathscr{D}([-a_j,a_j]) \,; \, \|\varphi\|_J\leq 1\},B_{j,J}\right),
$$
where $\|\varphi\|_J = \|\varphi\|_{[-a_j,a_j],J}$. By \eqref{eq: decomposition} and the previous equality, we obtain that 
	\begin{gather*}
		\forall M \in\N_0\, \exists K \in \N_0, C > 0 \,:\,
		W\left(\{\varphi\in\mathscr{D}([-a_{2},a_{2}])\, ; \,\|\varphi\|_{M}\leq 1\},B_{2,M}\right)\subseteq \\
		C\left(W\left(\{\varphi\in\mathscr{D}([-a_3,a_3])\, ; \,\|\varphi\|_K\leq 1\},B_{3,K}\right)  +W\left(\{\varphi\in\mathscr{D}([-a_1,a_1]) \,;\,\|\varphi\|_0\leq 1\},B_{1,0}\right)\right),
	\end{gather*}
	 This implies that (cf.\  \cite[Proof of necessity in Theorem 4.1]{B-D2006}) 
\begin{gather}\label{almost-there}
		\forall M\in\N_0\, \exists K \in \N_0, C > 0 \,\forall \varphi \in \mathscr{D}([-a_1,a_1]) \, :\, \\ \nonumber
	\|\varphi\|_MB_{2,M}\leq C\left(\|\varphi\|_KB_{3,K}+ \|\varphi\|_0 B_{1,0}\right).
\end{gather}
Let $M\in\N_0$ be arbitrary. We may assume that $M \geq 1$. Choose $\varphi \in \mathscr{D}([-a_1,a_1])$ such that $\varphi'(0) = 1$. Set $\varphi_\varepsilon(x) = \varphi(x/\varepsilon)$ for  $\varepsilon\in (0,1)$. Then, $\varphi_\varepsilon \in \mathscr{D}([-a_1,a_1])$, $\| \varphi_\varepsilon\|_M \geq |\varphi_\varepsilon'(0)| = \varepsilon^{-1}$, and $\| \varphi_\varepsilon\|_J \leq \| \varphi\|_J \varepsilon^{-J}$ for all $J \in \N_0$. The result now follows by setting $\varphi = \varphi_{\varepsilon}$, $\varepsilon\in (0,1)$, in  \eqref{almost-there}.
\end{proof}
We now  deduce from Proposition \ref{lem: explicit POmega} the  condition $(\Omega)$  for $\mathscr{E}_P(X)$ (see \eqref{def-Omega}) with control on the seminorms occurring in it. To this end, following \cite{DeKa22}, we use a simple cut-off and regularization procedure and a basic result in the theory of the derived projective limit functor (the Mittag-Leffler lemma), see the book \cite{Wengenroth} for more information on this topic. For $X \Subset \R^d$ we write $\mathscr{E}(\overline{X})$ for the space consisting of all $f \in \mathscr{E}(X)$ such that $\partial^\alpha f$ has a continuous extension to $\overline{X}$ for all $\alpha \in \N^d$.  We endow it with the family of norms $\{\| \, \cdot \, \|_{{\overline{X},J}} \, ; \, J \in \N_0\}$, hence it becomes a Fr\'echet space. We set
$$
\mathscr{E}_P(\overline{X}) = \{f \in \mathscr{E}(\overline{X}) \, ; \, P(D) f = 0\}
$$
and endow it with the subspace topology from $\mathscr{E}(\overline{X})$. Then,  $\mathscr{E}_P(\overline{X})$ is also a Fr\'echet space.

\begin{proposition}\label{cor: explicit Omega}
	Let $P\in\C[\xi_1,\ldots,\xi_d]$ and let $X \subseteq \R^d$ be open such that $P^+(D): \mathscr{D}'(X \times \R) \to \mathscr{D}'(X \times \R)$ is surjective. Let $X_1, X_2 \Subset X$ with $X_1 \subseteq X_2$ be such that $(\overline{X_1}, \overline{X_2})$ is augmentedly $P$-locating for $X$. Then, for all $X'_1 \Subset X''_1 \Subset X_1$, $X' \Subset X$, and  $r_1,r_2\in\N_0$  there exist $s, C>0$ such that
	\begin{gather*}
	\forall f \in\mathscr{E}_P(X),\,\varepsilon\in (0,1) \,\exists\,h_\varepsilon \in\mathscr{E}_P(X) \,:\, \\ 
		\|f-h_\varepsilon\|_{\overline{X'_1},r_1}\leq\varepsilon\max\{\|f\|_{\overline{X''_1}, r_1+1}, \|f\|_{\overline{X_2}}\}  \quad
		\mbox{ and } \quad \|h_\varepsilon\|_{\overline{X'},r_2}\leq\frac{C}{\varepsilon^s}\|f\|_{\overline{X_2}}.
\end{gather*}
In particular, it holds that $	\|f-h_\varepsilon\|_{\overline{X'_1},r_1}\ \leq \varepsilon\|f\|_{\overline{X_2}, r_1+1}$.
\end{proposition}
\begin{proof} Fix $X'_1 \Subset X''_1 \Subset X_1$.
	
	\textsc{Auxiliary Claim 1:} \emph{For all $X_1 \Subset X'_3 \Subset X$ and  $r_1,r_2\in\N_0$  there exist $s, C_1,C_2>0$ and $\varepsilon_0 \in (0,1)$ such that 
	\begin{gather*}
		 	\forall f\in\mathscr{E}_P(X), \varepsilon\in (0,\varepsilon_0) \,\exists\,g_\varepsilon \in\mathscr{E}_P(\overline{X'_3})\,: \\
		 \|f-g_\varepsilon\|_{\overline{X'_1},r_1}\leq C_1\varepsilon\max\{\|f\|_{\overline{X''_1}, r_1+1}, \|f\|_{\overline{X_2}}\} \quad
		\mbox{ and } \quad \|g_\varepsilon\|_{\overline{X'_3},r_2}\leq\frac{C_2}{\varepsilon^s}\|f\|_{\overline{X_2}}.
	\end{gather*}
	}
\emph{Proof of auxiliary claim 1.} Let $X_1 \Subset X'_3 \Subset X$ and $r_1,r_2\in\N_0$  be arbitrary. Choose $X_3 \Subset X$  such that $X_2 \subseteq X_3$ and $X'_3 \Subset X_3$. Set $\varepsilon_0 = \min \{1,d(\overline{X'_1}, \R^d \backslash X''_1), d(\overline{X'_3},\R^d \backslash X_3) \}$. By Proposition \ref{lem: explicit POmega} and a rescaling argument, there are  $K \in \N_0$ and $s',C >0$ such that 
$$
	| \overline{X}_2|B_{\overline{X}_2,0} \subseteq\frac{C}{\delta^{s'}} B_{\overline{X}_3,K}+\delta B_{\overline{X}_1,0}, \qquad \forall \delta \in (0,1),
$$
where $|\overline{X}_2|$ denotes the Lebesgue measure of $\overline{X}_2$. Let $f\in\mathscr{E}_P(X)$  be arbitrary. As $f \in | \overline{X_2}|\|f\|_{\overline{X_2}}B_{\overline{X}_2,0}$,  we have that for all $\delta \in (0,1)$ there is $f_\delta \in C\delta^{-s'}  \|f\|_{\overline{X_2}}B_{\overline{X}_3,K}$ such that $f - f_\delta \in \delta \|f\|_{\overline{X_2}}B_{\overline{X}_1,0}$. Choose $\chi \in \mathscr{D}(\R^d)$ with $\chi \geq 0$, $\operatorname{supp} \chi \subseteq B(0,1)$, and $\int_{\R^d} \chi(x) \dx = 1$, and set $\chi_\varepsilon = \varepsilon^{-d}\chi(x/\varepsilon)$ for $\varepsilon \in (0,1)$. For $\delta \in (0,1)$ and $\varepsilon \in (0,\varepsilon_0)$ we define $g_{\delta, \varepsilon} = f_\delta \ast \chi_\varepsilon \in \mathscr{E}_P(\overline{X'_3})$. The mean value theorem implies that
 $$
 \| f - f \ast \chi_\varepsilon\|_{\overline{X'_1}, r_1} \leq \sqrt{d} \varepsilon \|f\|_{\overline{X''_1}, r_1+1}, \qquad \forall \varepsilon \in (0,\varepsilon_0).
 $$
We write $\| \chi \|_{j} = \| \chi \|_{\overline{B}(0,1), j}$ for $j \in \N_0$. Since $f - f_\delta \in \delta \|f\|_{\overline{X_2}}B_{\overline{X}_1,0}$, it holds that
 $$
  \| f \ast \chi_\varepsilon - g_{\delta, \varepsilon}\|_{\overline{X'_1}, r_1} =   \| (f - f_\delta) \ast \chi_\varepsilon\|_{\overline{X'_1}, r_1} \leq \frac{\delta}{\varepsilon^{r_1 + d}} \| \chi \|_{r_1}\|f\|_{\overline{X_2}}, \qquad \forall \varepsilon \in (0,\varepsilon_0),
 $$
 Similarly, as $f_\delta \in C\delta^{-s'}  \|f\|_{\overline{X_2}}B_{\overline{X}_3,K}$, we have that
  $$
 \|  g_{\delta, \varepsilon}\|_{\overline{X'_3}, r_2}  \leq \frac{C}{\delta^{s'}\varepsilon^{K +r_2 + d}} \| \chi \|_{K + r_2}\|f\|_{\overline{X_2}}, \qquad \forall \varepsilon \in (0,\varepsilon_0).
 $$
 Define $g_\varepsilon =  g_{\varepsilon^{r_1+d + 1}, \varepsilon} \in \mathscr{E}_P(\overline{X'_3})$ for  $\varepsilon \in (0,\varepsilon_0)$. Set $s = s'(r_1 + d + 1) + K + r_2 + d$. Then, for all $\varepsilon \in (0,\varepsilon_0)$
 \begin{eqnarray*}
 \| f - g_\varepsilon\|_{\overline{X'_1}, r_1} &\leq& \| f -  f\ast \chi_\varepsilon\|_{\overline{X'_1}, r_1}
 + \| f \ast \chi_\varepsilon - g_\varepsilon\|_{\overline{X'_1}, r_1} \\
 &\leq&
 (\sqrt{d} + \| \chi\|_{r_1}) \varepsilon\max\{\|f\|_{\overline{X''_1}, r_1+1}, \|f\|_{\overline{X_2}}\},
 \end{eqnarray*}
 and
  $$
 \|  g_\varepsilon\|_{\overline{X'_3}, r_2} \leq \frac{C \| \chi\|_{K +r_2} }{\varepsilon^s}\|f\|_{\overline{X_2}}.
 $$
 	\textsc{Auxiliary Claim 2:} \emph{For all $X' \Subset X$ and  $r_1,r_2\in\N_0$  there exist $s, C_1,C_2>0$ and $\varepsilon_0 \in (0,1)$ such that 
 		\begin{gather*}
 			\forall f\in\mathscr{E}_P(X), \varepsilon\in (0,\varepsilon_0) \,\exists\,g_\varepsilon \in\mathscr{E}_P(X)\,: \\
 			\|f-g_\varepsilon\|_{\overline{X'_1},r_1}\leq C_1\varepsilon\max\{\|f\|_{\overline{X''_1}, r_1+1}, \|f\|_{\overline{X_2}}\} \quad
 			\mbox{ and } \quad \|g_\varepsilon\|_{\overline{X'},r_2}\leq\frac{C_2}{\varepsilon^s}\|f\|_{\overline{X_2}}.
 		\end{gather*}
 	}
\noindent Note that, by a simple rescaling argument, the auxiliary claim 2 implies the result.  

\emph{Proof of auxiliary claim 2.}
 Let $(\Omega_j)_{j \in \N_0}$ be an  exhaustion by relatively compact open subsets of $X$, i.e., $\Omega_j \Subset \Omega_{j+1} \Subset X$ for all $j \in \N_0$  and $X = \bigcup_{j \in \N_0} \Omega_j$. Since $P^+(D): \mathscr{D}'(X \times \R) \to \mathscr{D}'(X \times \R)$ is surjective, $P(D): \mathscr{E}(X) \to \mathscr{E}(X)$ is as well. Hence, by \cite[Lemma 3.1]{DeKa22} (see also \cite[Section 3.4.4]{Wengenroth}), we have that $\operatorname{Proj}^1( \mathscr{E}_P(\overline{\Omega}_j))_{j \in \N_0} = 0$.
 Let  $X' \Subset X$ and $r_1,r_2\in\N_0$  be arbitrary.  We may assume that $X'_1 \Subset X'$. Since $\operatorname{Proj}^1( \mathscr{E}_P(\overline{\Omega}_j))_{j \in \N_0} = 0$, the Mittag-Leffler lemma \cite[Theorem 3.2.8]{Wengenroth} implies that there is $X'_3 \Subset X$ such that
 \[
 \mathscr{E}_P(\overline{X'_3}) \subseteq \mathscr{E}_P(X) + \{ f \in \mathscr{E}_P(\overline{X'}) \, ; \, \| f \|_{\overline{X'},r_2} \leq 1 \}.
 \]
 We may assume that $X' \subseteq X'_3$.
 By multiplying both sides of the above inclusion with $\delta$, we find that
 \begin{equation}
 	\label{ML}
 \mathscr{E}_P(\overline{X'_3}) \subseteq \mathscr{E}_P(X) + \{ f \in \mathscr{E}_P(\overline{X'}) \, ; \, \| f \|_{\overline{X'},r_2} \leq \delta \}, \qquad \forall \delta > 0.
 \end{equation}
 Let $s,C_1,C_2, \varepsilon_0$ be as in the auxiliary claim 1 (with $X'_3 \Subset X$ and  $r_1,r_2\in\N_0$ as above). Let $f\in\mathscr{E}_P(X)$  be arbitrary. Choose $g_\varepsilon \in\mathscr{E}_P(\overline{X'_3})$, $\varepsilon\in (0,\varepsilon_0)$, as in the auxiliary claim 1. By \eqref{ML}, there is $h_\varepsilon \in \mathscr{E}_P(X)$ such that  $\|g_\varepsilon-h_\varepsilon\|_{\overline{X'},r_2}\leq \varepsilon  \|f\|_{\overline{X_2}}$ for all $\varepsilon\in (0,\varepsilon_0)$. Hence, for all $\varepsilon \in (0,\varepsilon_0)$
$$
 \| f - h_\varepsilon\|_{\overline{X'_1}, r_1} \leq \| f -   g_\varepsilon\|_{\overline{X'_1}, r_1}
 + \| g_\varepsilon - h_\varepsilon\|_{\overline{X'}, r_2} \leq
 (C_1 +1) \varepsilon\max\{\|f\|_{\overline{X''_1}, r_1+1}, \|f\|_{\overline{X_2}}\}, 
 $$
 and
 $$
 \|  h_\varepsilon\|_{\overline{X'}, r_2} \leq  \|  g_\varepsilon - h_\varepsilon\|_{\overline{X'}, r_2} + \|  g_\varepsilon\|_{\overline{X'_3}, r_2}  \leq \frac{C_2 +1}{\varepsilon^s} \|f\|_{\overline{X_2}}.
 $$
\end{proof}

For hypoelliptic operators it is more natural to work with sup-seminorms. In this regard, we have the following result.

\begin{corollary}\label{cor: explicit Omega-hypo}
	Let $P\in\C[\xi_1,\ldots,\xi_d]$  be hypoelliptic and let $X \subseteq \R^d$ be open such that $P^+(D): \mathscr{D}'(X \times \R) \to \mathscr{D}'(X \times \R)$ is surjective. Let $X_1, X_2 \Subset X$ with $X_1 \subseteq X_2$ be such that $(\overline{X_1}, \overline{X_2})$ is augmentedly $P$-locating for $X$. Then, for all $X'_1 \Subset X_1$ and $X' \Subset X$  there exist $s, C>0$ such that
	
	\begin{gather*}
		\forall f \in\mathscr{E}_P(X),\,\varepsilon\in (0,1) \,\exists\,h_\varepsilon \in\mathscr{E}_P(X) \,:\, \\ 
		\|f-h_\varepsilon\|_{\overline{X'_1}}\leq\varepsilon\|f\|_{\overline{X_2}} \quad
		\mbox{ and } \quad \|h_\varepsilon\|_{\overline{X'}}\leq\frac{C}{\varepsilon^s}\|f\|_{\overline{X_2}}.
	\end{gather*}
\end{corollary}
\begin{proof}
	Fix $X''_1 \Subset X_1$ such that $X'_1 \Subset X''_1$. Since $P(D)$ is hypoelliptic, there is $C' >0$ such that
	$$
	\|f\|_{\overline{X''_1},1} \leq C' \|f\|_{\overline{X_2}}, \qquad f \in \mathscr{E}_P(X).
	$$
The result now follows from  Proposition \ref{cor: explicit Omega} with $r_1 = r_2 = 0$.
\end{proof}

\section{Quantitative Runge type approximation theorems}\label{sec: quantitative Runge}
We now combine the results from Sections \ref{sec: qualitative Runge} and \ref{sec: technical} to obtain quantitative approximation results. In particular, we shall show Theorems \ref{theo: quantitative convex}-\ref{theo: quantitative Runge for wave operator} from the introduction. We start with the following general result.

\begin{proposition}\label{prop: general quantitative} 
	Let $P\in\C[\xi_1,\ldots,\xi_d]$ and let $X\subseteq\R^d$ be open such that $P^+(D): \mathscr{D}'(X\times\R) \to \mathscr{D}'(X\times\R)$ is surjective. Let $Y \subseteq X$ be open such that  the restriction map $r_{\mathscr{E}}^P:\mathscr{E}_P(X)\rightarrow\mathscr{E}_P(Y)$ has dense range.  Let $Y_1, Y_2 \Subset Y$ with $Y_1 \subseteq Y_2$ be such that $(\overline{Y_1}, \overline{Y_2})$ is augmentedly $P$-locating for $X$. Then, for all $Y'_1 \Subset Y_1$, $X' \Subset X$, and  $r_1,r_2\in\N_0$  there exist $s, C>0$ such that
	\begin{gather*}
		\forall f\in\mathscr{E}_P(Y), \varepsilon\in (0,1) \,\exists h_\varepsilon \in\mathscr{E}_P(X) \, :\\
		\|f-h_\varepsilon\|_{\overline{Y'_1},r_1}\leq \varepsilon\|f\|_{\overline{Y_2},r_1+1} \quad \mbox{ and } \quad \|h_\varepsilon\|_{\overline{X'},r_2}\leq\frac{C}{\varepsilon^s}\|f\|_{\overline{Y_2}}.
			\end{gather*}
\end{proposition}

\begin{proof} 
Let $Y'_1 \Subset Y_1$, $X' \Subset X$, and  $r_1,r_2\in\N_0$ be arbitrary.
	By Proposition \ref{cor: explicit Omega}, we find that there are $s,C >0$ such that
	\begin{gather}\label{eq: decomposition 1}
		\forall g\in\mathscr{E}_P(X),\varepsilon\in (0,1)\,\,\exists\,h_\varepsilon \in\mathscr{E}_P(X) \,: \\ \nonumber
		\|g-h_\varepsilon\|_{\overline{Y'_1},r_1}\leq \varepsilon\|g\|_{\overline{Y_2},r_1+1} \quad \mbox{ and } \quad \|h_\varepsilon\|_{\overline{X'},r_2}\leq\frac{C}{\varepsilon^s}\|g\|_{\overline{Y_2}}.
	\end{gather}
	Let $f\in\mathscr{E}_P(Y)$ and $\varepsilon\in (0,1)$ be arbitrary. Since $r_{\mathscr{E}}^P:\mathscr{E}_P(X)\rightarrow\mathscr{E}_P(Y)$ has dense range, there is $g_\varepsilon\in\mathscr{E}_P(X)$ with $\|f-g_\varepsilon\|_{\overline{Y_2}, r_1+1}\leq\varepsilon\|f\|_{\overline{Y_2}}$. Choose $h_\varepsilon$ according to \eqref{eq: decomposition 1} for $g = g_\varepsilon$. Then,
	\begin{eqnarray*}
		\|f-h_\varepsilon\|_{\overline{Y'_1}, r_1}&\leq&\|f-g_\varepsilon\|_{\overline{Y'_1}, r_1}+\|g_\varepsilon-h_\varepsilon\|_{\overline{Y'_1}, r_1}\\
		&\leq&\|f-g_\varepsilon\|_{\overline{Y_2}, r_1+1}+\varepsilon \|g_\varepsilon\|_{\overline{Y_2},r_1+1} \\
		&\leq&\|f-g_\varepsilon\|_{\overline{Y_2}, r_1+1}+\varepsilon\left(\|f-g_\varepsilon\|_{\overline{Y_2}, r_1+1}+\|f\|_{\overline{Y_2}, r_1+1}\right)\\
		&\leq& 3\varepsilon\|f\|_{\overline{Y_2},r_1+1},
	\end{eqnarray*}
and
	$$\|h_\varepsilon\|_{\overline{X'},r_2}\leq\frac{C}{\varepsilon^s}\|g_\varepsilon\|_{\overline{Y_2}}\leq \frac{C}{\varepsilon^s}\left(\|f-g_\varepsilon\|_{\overline{Y_2},r_1+1}+\|f\|_{\overline{Y_2}}\right)\leq\frac{2C}{\varepsilon^s}\|f\|_{\overline{Y_2}}.$$
	This implies the result.
\end{proof}
For hypoelliptic operators, we obtain the following result.

\begin{proposition}\label{cor: general quantitative-hypo}
	Let $P\in\C[\xi_1,\ldots,\xi_d]$ be hypoelliptic and let $X\subseteq\R^d$ be open such that $P^+(D): \mathscr{D}'(X\times\R) \to \mathscr{D}'(X\times\R)$ is surjective. Let $Y \subseteq X$ be open such that  the restriction map $r_{\mathscr{E}}^P:\mathscr{E}_P(X)\rightarrow\mathscr{E}_P(Y)$ has dense range.  Let $Y_1, Y_2 \Subset Y$ with $Y_1 \subseteq Y_2$ be such that $(\overline{Y_1}, \overline{Y_2})$ is augmentedly $P$-locating for $X$. Then, for all $Y'_1 \Subset Y_1$ and $X' \Subset X$  there exist $s, C>0$ such that
	\begin{gather*}
		\forall f\in\mathscr{E}_P(Y), \varepsilon\in (0,1) \,\exists h_\varepsilon \in\mathscr{E}_P(X) \, :\\
		\|f-h_\varepsilon\|_{\overline{Y'_1}}\leq \varepsilon\|f\|_{\overline{Y_2}} \quad \mbox{ and } \quad \|h_\varepsilon\|_{\overline{X'}}\leq\frac{C}{\varepsilon^s}\|f\|_{\overline{Y_2}}.
	\end{gather*}
\end{proposition}
\begin{proof}
This can be shown in the same way as Proposition \ref{prop: general quantitative}  but by using  Corollary \ref{cor: explicit Omega-hypo} instead of Proposition \ref{cor: explicit Omega}.
\end{proof}

As a first application of  Proposition \ref{prop: general quantitative} and Proposition \ref{cor: general quantitative-hypo}, we show Theorem \ref{theo: quantitative convex}.
\begin{proof}[Proof of Theorem \ref{theo: quantitative convex}.]
By possibly shrinking $Y$, we may assume that this set is convex. $P^+(D): \mathscr{D}'(\R^{d+1}) \to \mathscr{D}'(\R^{d+1})$ is surjective. By \cite[Chapitre 1.2, Th\'eor\`eme 2]{Malgrange} (see also  \cite[Theorem 10.5.1]{HoermanderPDO2}) the restriction map $r_{\mathscr{E}}^P:\mathscr{E}_P(\R^d)\rightarrow\mathscr{E}_P(Y)$ has dense range. Choose $Y' \Subset Y$ convex such that $K \subseteq Y'$ and $\overline{Y'} \subseteq L$. Then, $\overline{Y'}$ is augmentedly  $P$-locating for $\R^d$ by Lemma \ref{prop: location of sing/supp-convex}. The result now follows from Proposition \ref{prop: general quantitative} and Proposition \ref{cor: general quantitative-hypo}.
\end{proof}
Our next goal is to use Proposition \ref{prop: general quantitative} and Proposition \ref{cor: general quantitative-hypo} to obtain  quantified versions of Theorem \ref{theo: Runge for subspace elliptic} and Theorem \ref{theo: qualitative Runge for single characteristic direction}, and to show Theorem \ref{theo: quantitative Runge for wave operator}. We need the following simple geometrical lemma.

\begin{lemma}\label{prop: geometric}
	Let $V,X,Y \subseteq\R^d$  with $Y \subseteq X$ be  such that  $X$ does not contain a bounded connected component of $\left(\R^d\backslash Y\right)\cap V$. Let $K \subseteq Y$ and let $C$ be a bounded connected component of $\left(\R^d\backslash K\right)\cap V$. If $C \subseteq X$, then $C \subseteq Y$.
\end{lemma}

\begin{proof}
Suppose that $C \not \subseteq Y$. Let $x\in C \backslash Y$ and let $D$ be the connected component of $\left(\R^d\backslash Y\right)\cap V$ which contains $x$. Since $\left(\R^d\backslash Y\right)\cap V \subseteq \left(\R^d\backslash K\right)\cap V$, we have that $D \subseteq C$ and so $D$ is bounded. Hence, $D \not \subseteq X$ and thus $C \not \subseteq X$, a contradiction.
\end{proof}

\begin{theorem}\label{theo: quantitative Runge for subspace elliptic}
	Let $P\in\C[\xi_1,\ldots,\xi_d]$ be a polynomial which acts along a subspace $W$ of $\R^d$ and is elliptic on $W$. Let $X,Y \subseteq\R^d$ be open with $Y \subseteq X$ such that $X$ is $P$-convex for supports. Suppose that there is no $x\in\R^d$ such that $X$ contains a compact connected component of $(\R^d\backslash Y)\cap (x+W)$. Let $Y'\Subset Y$ be  such that there is no $x\in\R^d$ such that $Y$ contains a bounded connected component of  $(\R^d\backslash\overline{Y'})\cap (x+W)$. Then, for all $Y'' \Subset Y'$, $X' \Subset X$, and  $r_1,r_2\in\N_0$  there exist $s, C>0$ such that
	\begin{gather*}
		\forall f\in\mathscr{E}_P(Y), \varepsilon\in (0,1) \,\exists h_\varepsilon \in\mathscr{E}_P(X) \, :\\
		\|f-h_\varepsilon\|_{\overline{Y''},r_1}\leq \varepsilon\|f\|_{\overline{Y'},r_1+1} \quad \mbox{ and } \quad \|h_\varepsilon\|_{\overline{X'},r_2}\leq\frac{C}{\varepsilon^s}\|f\|_{\overline{Y'}}.
	\end{gather*}
\end{theorem}

\begin{proof}
	Since $X$ is $P$-convex for supports,  Theorem \ref{theo: cases of Omega}(a) yields that $P^+(D): \mathscr{D}'(X\times\R) \to \mathscr{D}'(X\times\R)$ is surjective. By Theorem \ref{theo: Runge for subspace elliptic}, the restriction map $r_\mathscr{E}^P:\mathscr{E}_P(X)\rightarrow\mathscr{E}_P(Y)$ has dense range.  Lemma \ref{prop: geometric} implies  that there is no $x\in\R^d$ such that $X$ contains a bounded connected component of $\left(\R^d\backslash\overline{Y'}\right)\cap\left(x+W\right)$, whence $\overline{Y'}$ is augmentedly $P$-locating for $X$ by Lemma \ref{lem: location of sing/supp-subspace}.  The result now  follows from Proposition \ref{prop: general quantitative}.
\end{proof}
 
\begin{theorem}\label{theo: quantitative Runge for elliptic-1}
	Let $P\in\C[\xi_1,\ldots,\xi_d]$ be elliptic. Let $X,Y \subseteq\R^d$ be open with $Y \subseteq X$. Suppose that $X$ contains no compact connected component of $\R^d\backslash Y$. Let $Y'\Subset Y$ be  such that $Y$ contains no bounded connected component of  $\R^d\backslash\overline{Y'}$. Then, for all $Y'' \Subset Y'$ and $X' \Subset X$  there exist $s, C>0$ such that
	\begin{gather}
		\label{approx-sect}
		\forall f\in\mathscr{E}_P(Y), \varepsilon\in (0,1) \,\exists h_\varepsilon \in\mathscr{E}_P(X) \, :\\ \nonumber
		\|f-h_\varepsilon\|_{\overline{Y''}}\leq \varepsilon\|f\|_{\overline{Y'}} \quad \mbox{ and } \quad \|h_\varepsilon\|_{\overline{X'}}\leq\frac{C}{\varepsilon^s}\|f\|_{\overline{Y'}}.
	\end{gather}
\end{theorem}

\begin{proof}
	Since $P$ is elliptic,  $P^+(D): \mathscr{D}'(X\times\R) \to \mathscr{D}'(X\times\R)$ is surjective (see Remark \ref{rem:augmented-elliptic}). By the Lax-Malgrange theorem, the restriction map $r_\mathscr{E}^P:\mathscr{E}_P(X)\rightarrow\mathscr{E}_P(Y)$ has dense range.  Lemma \ref{prop: geometric} implies that $X$ contains no bounded connected component of $\R^d\backslash \overline{Y'}$, whence $\overline{Y'}$ is augmentedly $P$-locating for $X$ by Lemma \ref{lem: location of sing/supp-subspace} with $W = \R^d$.  The result now  follows from Proposition \ref{cor: general quantitative-hypo}.
\end{proof}

\begin{proof}[Proof of  Theorem \ref{theo: quantitative Runge for wave operator}] Set $P_1(x_1,x_2)=x_1+x_2$ and $P_2(x_1,x_2)=x_2-x_1$, then $P=P_1P_2$.
		Since $X$ is $P$-convex for supports, Theorem \ref{theo: cases of Omega}(a) and Lemma \ref{prop: composition of differential operators}(a) imply that $P^+(D): \mathscr{D}'(X\times\R) \to \mathscr{D}'(X\times\R)$ is surjective. By Theorem \ref{theo: Runge for affine factors}, the restriction map $r_\mathscr{E}^P:\mathscr{E}_P(X)\rightarrow\mathscr{E}_P(Y)$ has dense range.   Lemma \ref{prop: geometric} implies  that there is no $x\in\R^d$ such that $X$ contains a bounded connected component of one of the sets  $\left(\R^d\backslash L \right)\cap \{(x_1+s,x_2+s)\,; \,s\in\R\}$ or  $\left(\R^d\backslash L \right)\cap \{(x_1-s,x_2+s)\,; \,s\in\R\}$.  Since $P_1$ and $P_2$ act along the subspaces $\mbox{span}\{(1,1)\}$ and $\mbox{span}\{(-1,1)\}$, respectively, and are elliptic there, $L$ is augmentedly $P_j$-locating for $X$, $j=1,2$, by Lemma \ref{lem: location of sing/supp-subspace}.
		This implies that $L$ is augmentedly $P$-locating for $X$. The result now follows from  Proposition \ref{prop: general quantitative}.	
\end{proof}
\begin{remark}
 More generally, taking into account Theorem \ref{theo: Runge for affine factors},  the same arguments used in the above proof yield a  quantitative approximation result for  partial differential operators that factor into first order operators. We leave the formulation and proof of this result to the reader.
\end{remark}

\begin{theorem}\label{theo: quantitative Runge for single characteristic direction}
		Let $P\in\C[\xi_1,\ldots,\xi_d]$ be semi-elliptic with principal part $P_m$ such that $\{P_m=0\}$ is a one-dimensional subspace of $\R^d$. Let $X,Y\subseteq\R^d$ be open with $Y \subseteq X$ such that $X$ is  $P$-convex for supports. Suppose that there is no $x\in\R^d$ such that $X$ contains a compact connected component of $(\R^d\backslash Y)\cap(x+\{P_m=0\}^\perp)$.  Let $Y'\Subset Y$ be  such that there is no $x\in\R^d$ such that $Y$ contains a bounded connected component of  $(\R^d\backslash\overline{Y'})\cap (x+\{P_m=0\}^\perp)$. Then, for all $Y'' \Subset Y'$ and $X' \Subset X$  there exist $s, C>0$ such that 	\eqref{approx-sect} holds.
\end{theorem}

\begin{proof}
		Since $X$ is $P$-convex for supports,  Theorem \ref{theo: cases of Omega}(b) yields that $P^+(D): \mathscr{D}'(X\times\R) \to \mathscr{D}'(X\times\R)$ is surjective. By Theorem \ref{theo: qualitative Runge for single characteristic direction}, the restriction map $r_\mathscr{E}^P:\mathscr{E}_P(X)\rightarrow\mathscr{E}_P(Y)$ has dense range.  Lemma \ref{prop: geometric} implies  that there is no $x\in\R^d$ such that $X$ contains a bounded connected component of $\left(\R^d\backslash\overline{Y'}\right)\cap\left(x+ \{P_m=0\}
		^\perp\right)$, whence $\overline{Y'}$ is augmentedly $P$-locating for $X$ by Proposition \ref{prop: location of sing/supp-onedim}.  The result now  follows from Proposition \ref{cor: general quantitative-hypo}.	
\end{proof}

We now use  Theorem \ref{theo: quantitative Runge for single characteristic direction} to formulate a quantitative approximation result for certain $k$-parabolic  operators  on tubular domains. As in the introduction, we denote the elements of $ \R^d = \R^{n+1}$ by $(t,x)$ with $t\in\R$, $x\in\R^n$.

\begin{corollary}\label{cor: quantitative Runge for parabolic}
	 Let $Q\in\C[\xi_1,\ldots,\xi_n]$ be an elliptic polynomial of degree $m$ with principal part $Q_m$, let $r\in\N$ with $m>r$, and let $\alpha\in\C\backslash\{0\}$. Suppose that $\{ Q_m(x) \, ; \, x \in \R^n\backslash\{0\} \} \cap \{\alpha t^r \,;\,t\in\R\} = \emptyset$. Consider $P(t,x)=\alpha t^r-Q(x)$.
	Let $G,H\subseteq\R^n$ be open with $H\subseteq G$ and suppose that $G$ does not contain a compact connected component of $\R^n\backslash H$. Let $J\subseteq\R$ be an open interval and let $Y'\Subset J\times H$ be such that there is no $c\in\R$ such that $J \times H$ contains a compact connected component of  $\left(\R^{n+1}\backslash\overline{Y'}\right) \cap\{(c, x) \, ; \,  x \in\R^{n}\}$. Then, for all $Y'' \Subset Y'$ and $X' \Subset \R\times G$  there exist $s, C>0$ such that
	\begin{gather*}
		\forall f\in\mathscr{E}_P(J \times H), \varepsilon\in (0,1) \,\exists h_\varepsilon \in\mathscr{E}_P(\R \times G) \, :\\
		\|f-h_\varepsilon\|_{\overline{Y''}}\leq \varepsilon\|f\|_{\overline{Y'}} \quad \mbox{ and } \quad \|h_\varepsilon\|_{\overline{X'}}\leq\frac{C}{\varepsilon^s}\|f\|_{\overline{Y'}}.
	\end{gather*}
\end{corollary}

\begin{proof} Note that $P$ is semi-elliptic with principal part $P_m(x,t) = Q_m(x)$. Hence, $\{P_m=0\}$ is a one-dimensional subspace of $\R^{n+1}$. By  \cite[Corollary 5]{Kalmes19}, $\R \times G$ is $P$-convex for supports. The result now follows from Theorem \ref{theo: quantitative Runge for single characteristic direction}.
\end{proof}

Finally, we deduce Theorem \ref{theo: quantitative Lax-Malgrange} and  Theorem \ref{theo: quantitative parabolic} from Theorem \ref{theo: Runge for subspace elliptic} and Corollary \ref{cor: quantitative Runge for parabolic}, respectively. To this end, we need a geometrical lemma. Recall that an open  set $U \subseteq \R^d$ is said to have continuous boundary if for each $\xi \in \partial U$ there exists a homeomorphism  $h$ from the closed unit ball $\overline{B}(0,1)$ in $\R^d$ onto a compact neighborhood $K$ of $\xi$ such that $h(0) = \xi$ and  $\overline{U} \cap h\left(\overline{B}(0,1)\right)= h\left(\{x\in \overline{B}(0,1)\, ; \,x_d\leq 0\}\right)$.
\begin{lemma}\label{lem: geometric} 
		Let $X\subseteq\R^d$ be open and let $U \Subset X$ have continuous boundary. Set $K = \overline{U}$. Suppose that $X$ contains no bounded connected component of $\R^d\backslash K$. Then, every open subset $Y$ of $\R^d$ with $K\subseteq Y\subseteq X$ contains an open set $Z$ with $K\subseteq Z$ such  that  $X$ contains no bounded connected component of $\R^d\backslash Z$.
\end{lemma}

\begin{proof}
	We set $\varepsilon_0= d(K,\R^d\backslash Y)/2$ if $Y\neq\R^d$ and $\varepsilon_0=1/2$ otherwise. For $\rho>0$ we abbreviate $K_\rho= K + B(0,\rho)$. Then, $K_{\varepsilon_0} \Subset Y$ and $\partial X\subseteq\R^d\backslash \overline{K_{\varepsilon_0}}$.

Before we proceed, we give a short outline of the proof. The idea is to construct $\R^d\backslash Z$ by considering the unbounded connected component of $\R^d\backslash K_{\varepsilon_0}$ together with  its connected components that intersect $\partial X$ and to successively enlarge these components. We will show that only a finite number of enlargement steps are needed to come close enough to the boundary of $K$. Here, "close enough" means that it is possible to take the union of this enlargement with a suitable open superset of $\R^d\backslash K$ without losing the property that bounded connected components intersect $\partial X$.

	Let $R>0$ be such that $K_{\varepsilon_0}\subseteq \overline{B}(0,R)$. Then, the single unbounded connected component  $\R^d\backslash K_{\varepsilon_0}$ contains $\R^d\backslash \overline{B}(0,R)$. Let $\mathscr{C}_0$ be the set consisting of all connected components $C$ of $\R^d\backslash K_{\varepsilon_0}$ that are unbounded or intersect $\partial X$. Set $B_1=\cup_{C\in\mathscr{C}_0}C$.
	
	\textsc{Auxiliary Claim 1:}  \emph{$(\R^d\backslash  \overline{B}(0,R)) \cup \partial X \subseteq B_1$ and $B_1$ is a closed subset of $\R^d$.}
	
	\emph{Proof of auxiliary claim 1}. The first assertion is clear. To show that $B_1$ is closed, let $(x_n)_{n\in\N}$ be a sequence in $B_1$ that converges to $x$. We write $C_\infty$ for the unbounded connectecd component of $\R^d\backslash K_{\varepsilon_0}$. We distinguish two cases. Suppose that infinitely many $x_n$ belong to $C_\infty$. Since $C_\infty$ is closed in $\R^d$, we obtain that $x \in C_\infty \subseteq B_1$.  If only finitely many $x_n$ belong to $C_\infty$, we may assume that for every $n\in\N$ the connected component $C_n$ of $\R^d\backslash K_{\varepsilon_0}$ that contains $x_n$ is bounded. Since $\R^d\backslash  \overline{B}(0,R) \subseteq C_\infty$ we have that $C_n\subseteq  \overline{B}(0,R)$ for all $n\in\N$. By the definition of $\mathscr{C}_0$, there is $z_n\in\partial X\cap C_n$ for all $n\in\N$. We may assume that $(z_n)_{n\in\N}$ converges to some 
	$z\in\partial X\cap  \overline{B}(0,R)\subseteq \left(\R^d\backslash \overline{K_{\varepsilon_0}}\right)$.  Let $C$ be the connected component of $\R^d\backslash \overline{K_{\varepsilon_0}}$ that contains $z$. As $\R^d\backslash \overline{K_{\varepsilon_0}}$ is open, there is $\delta>0$ such that $B(z,\delta)\subseteq C$. Since $(z_n)_{n\in\N}$ converges to $z$, we may assume that $z_n\in B(z,\delta)$ for all $n \in \N$. Hence,  $B(z,\delta)\subseteq C_n$ for each $n\in\N$ and thus  $C_1=C_n$ for all $n\in\N$. Since $C_1$ is closed in $\R^d$, it follows from $x_n\in C_n=C_1, n\in\N$, that $x\in C_1 \subseteq B_1$. 
	
	We set 
	$$\delta(x)=\min\left\{1,\frac{d(x,K)}{2}\right\}, \qquad x \in \R^d \backslash K.$$
	For $n\in\N$ we recursively define $B_{n+1}=\bigcup_{x\in B_n} \overline{B}(x,\delta(x))$.
	
	\textsc{Auxiliary Claim 2:} \emph{For each $n \in \N$ it holds that  $B_n\subseteq \R^d\backslash K_{\varepsilon_0/2^{n-1}}$, every bounded connected component of $B_n$ intersects $\partial X$, and $B_n$ is a closed subset of $\R^d$.}
	
	\emph{Proof of auxiliary claim 2.} For $n=1$ the first two assertions are clear, while the third one has been shown in auxiliary claim 1.  The assertions for general $n$ can be shown by induction, the details are left to the reader. 
	
	\textsc{Auxiliary Claim 3:} 
		\emph{There are $m\in\N$ and finitely many  continuous homeomorphism $h_j:\overline{B}(0,1)\rightarrow h_j(\overline{B}(0,1)) \subseteq \R^d$, $1\leq j\leq k$, such that
	\begin{itemize}
		\item[(a1)] $\forall\,1\leq j\leq k:\,B_m\cap h_j\left(B(0,1)\right) \neq\emptyset$,
		\item[(a2)] $\partial K\subseteq\cup_{j=1}^k h_j\left(B(0,1)\right)\subseteq\bigcup_{j=1}^k h_j\left(\overline{B}(0,1)\right)\subseteq K_{\varepsilon_0}$,
		\item[(a3)] $\forall\,1\leq j\leq k:\,K\cap h_j\left(\overline{B}(0,1)\right) =h_j\left(\{x\in \overline{B}(0,1)\,;\,x_d\leq 0\}\right)$.
	\end{itemize}}
	
	\emph{Proof of auxiliary claim 3.} Fix an arbitrary $\xi_0\in\partial K$. Since $K$ has  continuous boundary, there is a homeomorphism $h:\overline{B}(0,1)\rightarrow h_j(\overline{B}(0,1)) \subseteq \R^d$ such that $h(\xi) = 0$ 
	and $K\cap h\left(\overline{B}(0,1)\right)=h\left(\{x\in \overline{B}(0,1)\, ; \,x_d\leq 0\}\right)$. By possibly shrinking the ball $\overline{B}(0,1)$ and rescaling the function $h$, we may assume that  $h\left(\overline{B}(0,1)\right)\subseteq K_{\varepsilon_0}$. Set $V=h\left(B(0,1)\right)$. We now prove that there is $m\in\N$ with  $B_m\cap V\neq\emptyset$. 
	Note that $V\cap\left(\R^d\backslash K\right) \neq \emptyset$.
	Let $y_0\in V\cap\left(\R^d\backslash K\right)$ be arbitrary. We are going to show that $y_0 \in B_m$ for some $m \in \N$. Since $\R^d\backslash K$ is open, its connected components are pathwise connected. We distinguish two cases. Suppose that $y_0$ belongs to the unbounded connected component of $\R^d\backslash K$. Fix $z_0$ in the unbounded connected component of $B_1\subseteq\R^d\backslash K$. Then there is a continuous curve $\gamma:[0,1]\rightarrow \R^d\backslash K$ with $\gamma(0)=z_0\in B_1$ and $\gamma(1)=y_0$. If $y_0$ does not belong to the unbounded connected component of $\R^d\backslash K$, the hypothesis that $X$ does not contain any bounded connected component of $\R^d\backslash K$ implies that there is a continuous curve $\gamma:[0,1]\rightarrow\R^d\backslash K$ with $\gamma(0)=z_0\in\partial X\subseteq B_1$ and $\gamma(1)=y_0$.
	Hence, in either case, there is a continuous curve $\gamma:[0,1]\rightarrow\R^d\backslash K$ with $\gamma(0)\in B_1$ and $\gamma(1)=y_0$. Set
	$$\delta=\min\left\{1,\frac{d(\gamma([0,1]),K)}{2}\right\}>0.$$
	Since $\gamma$ is uniformly continuous, there is $N \in \N$ such that $|\gamma(r)-\gamma(t)|<\delta$ for all  $r,t\in[0,1]$ with $|r-t|\leq 1/N$. As $\gamma(0)\in B_1$ and $\delta<\delta(\gamma(0))$, it holds that $ \overline{B}(\gamma(0),\delta)\subseteq B_2$. By the choice of $N$, we have that $\gamma(1/N)\in  \overline{B}(\gamma(0),\delta)\subseteq B_2$.  Likewise, as $\delta\leq\delta(\gamma(1/N))$, it holds that $\gamma(2/N)\in  \overline{B}(\gamma(1/N),\delta)\subseteq B_3$. Continuing this process,  we find that $y_0=\gamma(1)\in B_{N+1}$.  Since  $\xi_0\in\partial K$ was chosen arbitrarily, the claim now follows from the compactness of $\partial K$ and the fact that $B_n\subseteq B_{n+1}$ for all $n\in\N$.
	
	We can now prove the assertion of the lemma. To this end, let $m$ and $h_1,\ldots,h_k$ be as in the auxiliary claim 3. Choose $\rho\in (0,1)$ such that
	\begin{equation}
		\label{inclusionepsilon} 
	\bigcup_{j=1}^k h_j\left(\{x\in \overline{B}(0,1)\,;\, x_d\leq \rho\}\right)\subseteq K_{\varepsilon_0/2^{m-1}}.
	\end{equation}
	Set $Z' :=K\cup\bigcup_{j=1}^k h_j\left(\{x\in B(0,1)\,; \, x_d <\rho\}\right)$. Then, $Z'$ is  open and  $K \subseteq Z' \subseteq K_{\varepsilon_0/2^{m-1}} \subseteq Y$. Let $C$ be an arbitrary connected component of $\R^d\backslash Z'$. Then,
	$$\emptyset\neq C\cap\partial Z'\subseteq C\cap\left(\bigcup_{j=1}^k h_j\left(\{x\in \overline{B}(0,1)\,; \, x_d\leq \rho\}\right)\right),$$
This implies that there is $j \in \{1, \ldots, k\}$ such that
	$$\emptyset\neq C\cap h_j\left(\{x\in \overline{B}(0,1)\,; \,x_d= \rho\}\right)\subseteq C\cap h_j\left(\{x\in \overline{B}(0,1)\,;\,x_d\geq \rho\}\right).$$
As the union of two connected sets with non-empty intersection  is again connected, we obtain that $C\cup h_j\left(\{x\in \overline{B}(0,1)\,;\,x_d\geq \rho\}\right)$ is connected. Additionally, by condition (a1) and the inclusion $B_m \subseteq \R^d \backslash K_{\varepsilon_0/2^{m-1}}$, \eqref{inclusionepsilon} implies that there is a connected component $D$ of $B_m$ such that $D \cap h_j\left(\{x\in \overline{B}(0,1)\, ; \,x_d\geq \rho\}\right) \neq \emptyset$. Hence, $C\cup h_j\left(\{x\in \overline{B}(0,1)\, ; \,x_d\geq \rho\}\right)\cup D$ is connected. This implies that each of the  connected components of the set 
	$$ A = \left(\R^d\backslash Z'\right)\cup\bigcup_{j=1}^k h_j\left(\{x\in \overline{B}(0,1)\,;\,x_d\geq \rho\}\right)\cup B_m$$
	 contains a connected component of $B_m$ and therefore is unbounded or intersects $\partial X$ (auxiliary claim 2). Furthermore, this set is closed in $\R^d$ ($B_m$ is closed by the auxiliary claim 2).
    Finally, we define $Z = \R^d \backslash A$. Then, $Z$ is open, $Z \subseteq Z' \subseteq Y$ and each bounded connected component of $\R^d \backslash Z = A$ intersects $\partial X$. We also have that $K \subseteq Z$ because  $K \subseteq Z'$, $K \subseteq K_{\varepsilon_0/2^{m-1}} \subseteq \R^d \backslash B_m$   and the fact that, by (a3), $K \subseteq \R^d \backslash h_j\left(\{x\in \overline{B}(0,1)\,;\,x_d\geq \rho\}\right)$ for all $j = 1, \ldots, k$.	
\end{proof}

\begin{proof}[Proof of Theorem \ref{theo: quantitative Lax-Malgrange}] 
	By Lemma \ref{lem: geometric}, we may assume that $X$ contains  no bounded connected component of $\R^d\backslash Y$. Another application of Lemma \ref{lem: geometric} with $\operatorname{int} L$ in place of $Y$ shows that there is $Y' \Subset Y$ with $K \subseteq Y'$ and $\overline{Y'} \subseteq L$ such that $X$ contains no bounded connected component of $\R^d\backslash \overline{Y'}$. In particular, $Y$ contains no bounded connected component of $\R^d\backslash \overline{Y'}$. The result now follows from Theorem \ref{theo: quantitative Runge for elliptic-1}.
\end{proof}

\begin{proof}[Proof of Theorem \ref{theo: quantitative parabolic}] Let $L \subseteq H$ compact with $K \subseteq \operatorname{int} L$ and $\delta > 0$ with $[t_1-\delta,t_2+\delta]\subseteq J$ be arbitrary. 
	By Lemma \ref{lem: geometric}, we may assume that $G$ contains  no bounded connected component of $\R^n\backslash H$. Another application of Lemma \ref{lem: geometric} shows that there is $H' \Subset H$ with $K \subseteq H'$ and $\overline{H'} \subseteq L$ such that $H$ contains no bounded connected component of $\R^n\backslash \overline{H'}$, whence there is  no $c\in\R$ such that $J \times H$ contains a compact connected component of $\left(\R^{n+1}\backslash(J'\times H')\right)\cap\{(c, x) \,; \, x \in\R^{n}\}$. The result now follows from Corollary \ref{cor: quantitative Runge for parabolic}.
\end{proof}

\section*{Acknowledgements} We would like to thank the anonymous referee for suggesting to mention the results from \cite[Section 10.5]{HoermanderPDO2} and to use them to show that in all our global approximation results, i.e., when $X=\R^d$, the approximating function $h_\varepsilon$ can be chosen to be a linear combination of exponential solutions (cf.\ the introduction).

\end{document}